\documentclass[reqno, 12pt]{amsart}
\pdfoutput=1
\makeatletter
\let\origsection=\section \def\section{\@ifstar{\origsection*}{\mysection}}
\def\mysection{\@startsection{section}{1}\z@{.7\linespacing\@plus\linespacing}{.5\linespacing}{\normalfont\scshape\centering\S}}
\makeatother

\usepackage{amsmath,amssymb,amsthm}
\usepackage{mathrsfs}
\usepackage{mathabx}\changenotsign
\usepackage{dsfont}
\usepackage{bbm}

\usepackage{xcolor}
\usepackage[backref=section]{hyperref}
\usepackage[ocgcolorlinks]{ocgx2}
\hypersetup{
	colorlinks=true,
	linkcolor={red!60!black},
	citecolor={green!60!black},
	urlcolor={blue!60!black},
}

\usepackage[open,openlevel=2,atend]{bookmark}

\usepackage[abbrev,msc-links,backrefs]{amsrefs}
\usepackage{doi}

\renewcommand{\PrintDOI}[1]{\doi{#1}}

\usepackage[T1]{fontenc}
\usepackage{lmodern}
\usepackage[babel]{microtype}
\usepackage[english]{babel}

\usepackage{geometry}
\numberwithin{equation}{section}
\numberwithin{figure}{section}

\usepackage{enumitem}

\def\alabel{\upshape({\itshape \alph*\,})}

\let\polishlcross=\l
\def\l{\ifmmode\ell\else\polishlcross\fi}

\let\emptyset=\varnothing
\let\setminus=\smallsetminus

\makeatletter
\def\moverlay{\mathpalette\mov@rlay}
\def\mov@rlay#1#2{\leavevmode\vtop{   \baselineskip\z@skip \lineskiplimit-\maxdimen
		\ialign{\hfil$\m@th#1##$\hfil\cr#2\crcr}}}
\newcommand{\charfusion}[3][\mathord]{
	#1{\ifx#1\mathop\vphantom{#2}\fi
		\mathpalette\mov@rlay{#2\cr#3}
	}
	\ifx#1\mathop\expandafter\displaylimits\fi}
\makeatother

\newcommand{\dcup}{\charfusion[\mathbin]{\cup}{\cdot}}
\newcommand{\bigdcup}{\charfusion[\mathop]{\bigcup}{\cdot}}

\DeclareFontFamily{U}  {MnSymbolC}{}
\DeclareSymbolFont{MnSyC}         {U}  {MnSymbolC}{m}{n}
\DeclareFontShape{U}{MnSymbolC}{m}{n}{
	<-6>  MnSymbolC5
	<6-7>  MnSymbolC6
	<7-8>  MnSymbolC7
	<8-9>  MnSymbolC8
	<9-10> MnSymbolC9
	<10-12> MnSymbolC10
	<12->   MnSymbolC12}{}
\DeclareMathSymbol{\powerset}{\mathord}{MnSyC}{180}
\DeclareMathSymbol{\YY}{\mathord}{MnSyC}{42}

\usepackage{tikz}
\usetikzlibrary{calc,decorations.pathmorphing}
\usetikzlibrary{arrows,decorations.pathreplacing}
\pgfdeclarelayer{background}
\pgfdeclarelayer{foreground}
\pgfdeclarelayer{front}
\pgfsetlayers{background,main,foreground,front}
\definecolor{uuuuuu}{rgb}{0.27,0.27,0.27}
\definecolor{sqsqsq}{rgb}{0.1255,0.1255,0.1255}

\usepackage{multicol}
\usepackage{subcaption}
\let\epsilon=\varepsilon
\let\eps=\epsilon
\let\phi=\varphi
\let\rho=\varrho
\let\theta=\vartheta

\def\NN{{\mathds N}}

\def\PP{{\mathds P}}
\def\RR{{\mathds R}}

\def\ex{{\mathrm{ex}}}

\newcommand{\cD}{\mathcal{D}}

\newcommand{\cF}{\mathcal{F}}
\newcommand{\cG}{\mathcal{G}}
\newcommand{\cH}{\mathcal{H}}

\newcommand{\cK}{\mathcal{K}}

\newcommand{\cM}{\mathcal{M}}

\newcommand{\cS}{\mathcal{S}}
\newcommand{\cT}{\mathcal{T}}

\newcommand{\gM}{\mathfrak{M}}

\newtheoremstyle{note}  {4pt}  {4pt}  {\sl}  {}  {\bfseries}  {.}  {.5em}          {}
\newtheoremstyle{introthms}  {3pt}  {3pt}  {\itshape}  {}  {\bfseries}  {.}  {.5em}          {\thmnote{#3}}
\newtheoremstyle{remark}  {2pt}  {2pt}  {\rm}  {}  {\bfseries}  {.}  {.3em}          {}

\theoremstyle{plain}
\newtheorem{theorem}{Theorem}[section]
\newtheorem{lemma}[theorem]{Lemma}

\newtheorem{corollary}[theorem]{Corollary}
\newtheorem{observation}[theorem]{Observation}
\newtheorem{proposition}[theorem]{Proposition}
\newtheorem{problem}[theorem]{Problem}
\newtheorem{constr}[theorem]{Construction}

\newtheorem{fact}[theorem]{Fact}
\newtheorem{claim}[theorem]{Claim}

\theoremstyle{note}

\newtheorem{definition}[theorem]{Definition}

\theoremstyle{remark}

\usepackage{lineno}
\newcommand*\patchAmsMathEnvironmentForLineno[1]{
	\expandafter\let\csname old#1\expandafter\endcsname\csname #1\endcsname
	\expandafter\let\csname oldend#1\expandafter\endcsname\csname end#1\endcsname
	\renewenvironment{#1}
	{\linenomath\csname old#1\endcsname}
	{\csname oldend#1\endcsname\endlinenomath}}
\newcommand*\patchBothAmsMathEnvironmentsForLineno[1]{
	\patchAmsMathEnvironmentForLineno{#1}
	\patchAmsMathEnvironmentForLineno{#1*}}
\AtBeginDocument{
	\patchBothAmsMathEnvironmentsForLineno{equation}
	\patchBothAmsMathEnvironmentsForLineno{align}
	\patchBothAmsMathEnvironmentsForLineno{flalign}
	\patchBothAmsMathEnvironmentsForLineno{alignat}
	\patchBothAmsMathEnvironmentsForLineno{gather}
	\patchBothAmsMathEnvironmentsForLineno{multline}
}

\usepackage{scalerel}

\makeatletter
\newcommand{\overrighharpoonup}[1]{\ThisStyle{%
		\vbox {\m@th\ialign{##\crcr
				\rightharpoonupfill \crcr
				\noalign{\kern-\p@\nointerlineskip}
				$\hfil\SavedStyle#1\hfil$\crcr}}}}

\def\rightharpoonupfill{%
	$\SavedStyle\m@th\mkern+0.8mu\cleaders\hbox{$\shortbar\mkern-4mu$}\hfill\rightharpoonuptip\mkern+0.8mu$}

\def\rightharpoonuptip{%
	\raisebox{\z@}[2pt][1pt]{\scalebox{0.55}{$\SavedStyle\rightharpoonup$}}}

\def\shortbar{%
	\smash{\scalebox{0.55}{$\SavedStyle\relbar$}}}
\makeatother

\let\lra=\longrightarrow

\makeatletter
\newsavebox\myboxA
\newsavebox\myboxB
\newlength\mylenA

\newcommand*\xoverline[2][0.75]{%
	\sbox{\myboxA}{$\m@th#2$}%
	\setbox\myboxB\null
	\ht\myboxB=\ht\myboxA%
	\dp\myboxB=\dp\myboxA%
	\wd\myboxB=#1\wd\myboxA
	\sbox\myboxB{$\m@th\overline{\copy\myboxB}$}
	\setlength\mylenA{\the\wd\myboxA}
	\addtolength\mylenA{-\the\wd\myboxB}%
	\ifdim\wd\myboxB<\wd\myboxA%
	\rlap{\hskip 0.5\mylenA\usebox\myboxB}{\usebox\myboxA}%
	\else
	\hskip -0.5\mylenA\rlap{\usebox\myboxA}{\hskip 0.5\mylenA\usebox\myboxB}%
	\fi}
\makeatother

\DeclareSymbolFont{symbolsC}{U}{txsyc}{m}{n}
\SetSymbolFont{symbolsC}{bold}{U}{txsyc}{bx}{n}
\DeclareFontSubstitution{U}{txsyc}{m}{n}
\DeclareMathSymbol{\strictif}{\mathrel}{symbolsC}{74}

\DeclareSymbolFont{stmry}{U}{stmry}{m}{n}
\DeclareMathSymbol\arrownot\mathrel{stmry}{"58}
\DeclareMathSymbol\Arrownot\mathrel{stmry}{"59}

\begin{document}
\title[Hypergraphs with many extremal configurations]{Hypergraphs with many extremal configurations}

\author{Xizhi Liu}
\address{Department of Mathematics, Statistics, and Computer Science, University of Illinois,
Chicago, IL 60607 USA}
\email{xliu246@uic.edu}
\thanks{The first and second author's research is partially supported by NSF awards 
DMS-1763317 and DMS-1952767.}

\author{Dhruv Mubayi}
\address{Department of Mathematics, Statistics, and Computer Science, University of Illinois,
Chicago, IL 60607 USA}
\email{mubayi@uic.edu}

\author{Christian Reiher}
\address{Fachbereich Mathematik, Universit\"at Hamburg, Hamburg, Germany}
\email{Christian.Reiher@uni-hamburg.de}

\subjclass[2010]{}
\keywords{hypergraph Tur\'an problems, stability, feasible regions}

\begin{abstract}
For every positive integer $t$ we construct a finite family of triple systems $\cM_t$,
determine its Tur\'{a}n number, and show that there are $t$ extremal $\cM_t$-free configurations
that are far from each other in edit-distance.
We also prove a strong stability theorem: every $\cM_t$-free triple system whose size is close to 
the maximum size is a subgraph of one of these $t$ extremal configurations after removing a small 
proportion of vertices. This is the first stability theorem for a hypergraph problem with an 
arbitrary (finite) number of extremal configurations. Moreover, the extremal hypergraphs have 
very different shadow sizes (unlike the case of the famous Tur\'an tetrahedron conjecture). 
Hence a corollary of our main result is that the boundary of the feasible region of $\cM_t$ has 
exactly $t$ global maxima.
\end{abstract}

\maketitle

\section{Introduction}
\subsection{Stability}
Let $r \ge 2$ and let $\cF$ be a family of $r$-uniform hypergraphs (henceforth called $r$-graphs).
An $r$-graph $\cH$ is \emph{$\cF$-free} if it contains no member of $\cF$ as a subgraph.
For every natural number~$n$ the \emph{Tur\'{a}n number} ${\rm ex}(n,\cF)$ of $\cF$ is the 
maximum number of edges in an $\cF$-free $r$-graph on $n$ vertices. 
The \emph{Tur\'{a}n density} $\pi(\cF)$ of $\cF$ is defined 
as $\pi(\cF) = \lim_{n \to \infty} {\rm ex}(n,\cF) / \binom{n}{r}$, and $\cF$ is \emph{nondegenerate}
if $\pi(\cF) > 0$. By a theorem of Erd\H{o}s~\cite{E64}, this is equivalent to $\cF$ containing an 
$r$-graph which is not $r$-partite. 

The study of ${\rm ex}(n,\cF)$ is perhaps the central topic in extremal graph and hypergraph theory.
Curiously, unlike the case for graphs, determining $\pi(\cF)$ for a family $\cF$ of hypergraphs
is known to be notoriously hard in general.
Indeed, the problem of determining $\pi(K_{\ell}^{r})$ raised by Tur\'{a}n~\cite{TU41},
where $K_{\ell}^{r}$ is the complete $r$-graph on $\ell$ vertices, is still wide open for 
all $\ell>r\ge 3$. Erd\H{o}s offered $\$ 500$ for the determination of any $\pi(K_{\ell}^{r})$ 
with $\ell > r \ge 3$ and $\$ 1000$ for the determination of all $\pi(K_{\ell}^{r})$ 
with $\ell > r \ge 3$.

The classical Erd\H os-Simonovits stability theorem~\cite{SI68}
motivated the second author~\cite{MU07} to make the following definition.
A family $\cF$ of $r$-graphs is $t$-{\em stable} if for every $m\in\NN$
there exist $r$-graphs $\cG_{1}(m),\ldots,\cG_{t}(m)$ on $m$ vertices
such that the following holds.
For every $\delta >0$ there exist $\epsilon > 0$ and $n_0$ such that for all $n \ge n_0$
if $\cH$ is an $\cF$-free $r$-graph on~$n$ vertices with
\begin{align}
|\cH| > (1-\epsilon) {\rm ex}(n,\cF), \notag
\end{align}
then $\cH$ can be transformed to some $\cG_{i}(n)$ by adding and removing at most $\delta |\cH|$ edges.
Say~$\cF$ is {\em stable} if it is $1$-stable.
Denote by $\xi(\cF)$ the minimum integer $t$ such that $\cF$ is $t$-stable,
and set $\xi(\cF) = \infty$ if there is no such $t$.
Call $\xi(\cF)$ the \emph{stability number} of $\cF$.

The Erd\H{o}s-Stone-Simonovits theorem~\cites{ES46,ES66} and Erd\H{o}s-Simonovits stability 
theorem~\cite{SI68}
imply that every nondegenerate family of graphs is stable.
However, for hypergraphs there are many families (whose Tur\'{a}n densities are unknown)
which are conjecturally not stable. Two famous  examples are
Tur\'{a}n's conjecture on tetrahedra (e.g. see~\cites{AS95,KO82,LM2}) and
the Erd\H{o}s-S\'{o}s conjecture on triple systems with bipartite links (e.g. see~\cites{FF84,LM2}).
In fact, no Tur\'{a}n density of a nondegenerate family of hypergraphs without the stability property 
was known (e.g. see~\cite{KE11}) until very recently, when the first two authors constructed 
a $2$-stable family $\cM$ of triple systems~\cite{LM2}. Our first main result states that, 
more generally, for every natural number $t$ there exists a family of triple systems 
satisfying $\xi(\cM_t)=t$.
 
We identify an $r$-graph $\cH$ with its edge set,
use $V(\cH)$ to denote its vertex set, and denote by $v(\cH)$ the size of $V(\cH)$.
An $r$-graph $\cH$ is a \emph{blow-up} of an $r$-graph $\cG$ if there exists a 
map $\psi\colon V(\cH) \to V(\cG)$
so that $\psi(E) \in \cG$ iff $E\in \cH$,
and we say $\cH$ is \emph{$\cG$-colorable}
if there exists a map $\phi\colon V(\cH) \to V(\cG)$ so that $\phi(E)\in \cG$ for all $E\in \cH$.
In other words, $\cH$ is $\cG$-colorable if and only if $\cH$ occurs as a subgraph in some blow-up 
of $\cG$.

\begin{theorem}\label{THM:t-stability}
For every positive integer $t$ there exist constants $0 < n_1< \cdots < n_t$, $0 < \lambda_t < 1/6$,
$t$ triple systems $\cG_1,\ldots, \cG_t$ with $v(\cG_i) = n_i$ for $i\in[t]$,
and a finite family $\cM_t$ of triple systems with the following properties.
\begin{enumerate}[label=\alabel]
\item\label{it:11a} The inequality ${\rm ex}\left(n,\cM_t \right) \le \lambda_t n^3$ holds 
	for all positive integers $n$, and moreover, equality holds whenever $n$ is a multiple 
	of $n_i$ for some $i\in[t]$.
\item\label{it:11b} For every $\delta > 0$ there exist $\epsilon>0$ and $N_0$ so that the 
	following holds for all~$n \ge N_0$. Every $\cM_t$-free triple system $\cH$ on $n$ vertices 
	with at least $(\lambda_t-\epsilon) n^3$ edges can be made $\cG_i$-colorable for some $i\in [t]$
	by removing at most $\delta n$ vertices. Moreover,~$\xi(\cM_t) = t$.
\end{enumerate}
\end{theorem}

\subsection{Feasible regions}\label{subsec:feasible}
Recall that the \emph{shadow} of an $r$-graph $\cH$ is defined to be the $(r-1)$-graph
\begin{align}
\partial\cH 
= 
\left\{A\in \binom{V(\cH)}{r-1}\colon \text{there is } B\in \cH \text{ such that } 
	A\subseteq B\right\}. \notag
\end{align}
Call $d(\cH)= |\cH|/\binom{v(\cH)}{r}$ the \emph{edge density}
and $d(\partial\cH)= |\partial\cH|/\binom{v(\cH)}{r-1}$
the \emph{shadow density} of~$\cH$.

Given a family $\cF$ the \emph{feasible region} $\Omega(\cF)$ of $\cF$
is the set of points $(x,y)\in [0,1]^2$ such that there exists a sequence of $\cF$-free $r$-graphs
$\left( \cH_{k}\right)_{k=1}^{\infty}$ with $\lim_{k \to \infty}v(\cH_{k}) = \infty$,
$\lim_{k \to \infty}d(\partial\cH_{k}) = x$, and $\lim_{k \to \infty}d(\cH_{k}) = y$.
The feasible region unifies and generalizes many classical problems
such as the  Kruskal-Katona theorem~\cites{KR63,KA66} and the Tur\'{a}n problem. It
was introduced recently in \cite{LM1} to understand the extremal properties of $\cF$-free hypergraphs
beyond just the determination of $\pi(\cF)$. The general shape of $\Omega(\cF)$ was analyzed 
in~\cite{LM1} as follows: For some constant $c(\cF)\in [0, 1]$ the projection to the first 
coordinate, 
\[
	{\rm proj}\Omega(\cF) 
	= 
	\left\{ x \colon  \text{there is $y \in [0,1]$ such that $(x,y) \in \Omega(\cF)$} \right\},
\]
is the interval $[0, c(\cF)]$ . 
Moreover, there is a left-continuous almost everywhere differentiable function
$g(\cF)\colon {\rm proj}\Omega(\cF) \to [0,1]$ such that 
\[
	\Omega(\cF)
	=
	\bigl\{(x, y)\in [0, c(\cF)]\times [0, 1]\colon 0\le y\le g(\cF)(x)\bigr\}\,.
\]
Let us call $g(\cF)$ the {\it feasible region function} of $\cF$. 
There are examples showing that~$g(\cF)$ is not necessarily 
continuous (see~\cite{LM1}*{Theorem 1.12})
and the present work is part of an effort to figure out how ``exotic'' these functions can be.  

The stability number of $\cF$ can give information about the shape of $\Omega(\cF)$,
more precisely, about the number of global maxima of $g(\cF)$ (e.g. see 
Proposition~\ref{PROP:number-of-maxima-and-stability-number}).
The family $\cM$ of triple systems from~\cite{LM2} for which $\xi(\cM)=2$ has the following 
additional property: not only are the two near extremal constructions for $\cM$ far from each 
other in edit-distance, but the same is true of their shadows.
As a consequence, in addition to $\xi(\cM) = 2$, the function~$g(\cM)$ has exactly two global maxima.
The authors raised the question of whether there exists a finite family $\cM_t$ of triple systems 
so that the function $g(\cM_t)$ has exactly $t$ global maxima for $t \ge 3$ 
(see~\cite{LM1}*{Problem~6.10}). Our second main result 
asserts that the objects constructed in the course of proving Theorem~\ref{THM:t-stability}
give a positive solution to this problem. 

\begin{theorem}\label{THM:feasible-region-Mt}
For every positive integer $t$ there exist constants $0 < n_1< \cdots < n_t$, $0 < \lambda_t < 1/6$,
and a finite family $\cM_t$ of triple systems such that ${\rm proj}\Omega(\cM_t) = [0,1]$,
and $g(\cM_t,x) \le 6\lambda_t$ for all $x \in [0,1]$.
Moreover, $g(\cM_t,x) = 6\lambda_t$ if and only if $x = 1-1/n_i$ for some $i\in[t]$.
\end{theorem}

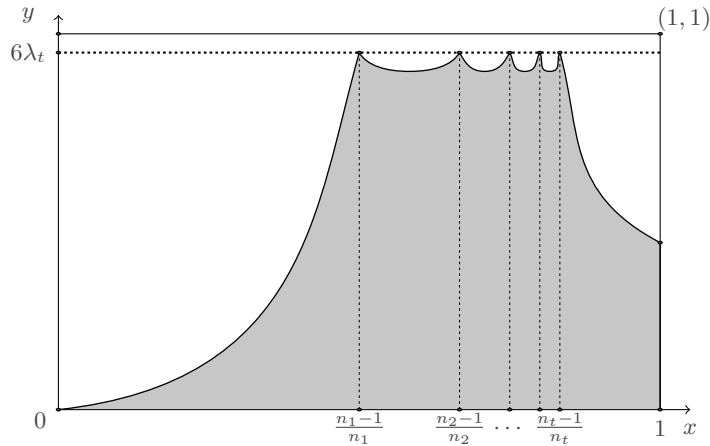
\begin{figure}[htbp]
\centering
\begin{tikzpicture}[xscale=8,yscale=5]
\draw [->] (0,0)--(1+0.05,0);
\draw [->] (0,0)--(0,1+0.05);
\draw (0,1)--(1,1);
\draw (1,0)--(1,1);
\draw [line width=0.8pt,dash pattern=on 1pt off 1.2pt,domain=0:1] plot(\x,1-0.05);
\draw [line width=0.4pt,dash pattern=on 1pt off 1.2pt] (1/2,0) -- (1/2,1-0.05);
\draw [line width=0.4pt,dash pattern=on 1pt off 1.2pt] (2/3,0) -- (2/3,1-0.05);
\draw [line width=0.4pt,dash pattern=on 1pt off 1.2pt] (3/4,0) -- (3/4,1-0.05);
\draw [line width=0.4pt,dash pattern=on 1pt off 1.2pt] (4/5,0) -- (4/5,1-0.05);
\draw [line width=0.4pt,dash pattern=on 1pt off 1.2pt] (5/6,0) -- (5/6,1-0.05);

\draw[fill=sqsqsq,fill opacity=0.25, line width=0.5pt]
(0,0)  to [out = 10, in = 260] (1/2,1-0.05) to [out = 290, in = 180] (7/12,1-0.1) to [out = 0, in = 250] (2/3, 1-0.05)
to [out = 290, in = 180] (17/24,1-0.1) to [out = 0, in = 250] (3/4, 1-0.05)
to [out = 290, in = 180] (31/40,1-0.1) to [out = 0, in = 250] (4/5, 1-0.05)
to [out = 290, in = 180] (49/60,1-0.1) to [out = 0, in = 250] (5/6, 1-0.05)
to [out = 280, in = 140] (1,4/9) to (1,0);

\begin{scriptsize}
\draw [fill=uuuuuu] (1,0) circle (0.1pt);
\draw[color=uuuuuu] (1,0-0.05) node {$1$};
\draw [fill=uuuuuu] (0,0) circle (0.1pt);
\draw[color=uuuuuu] (0-0.03,0-0.03) node {$0$};
\draw [fill=uuuuuu] (0,1) circle (0.1pt);
\draw[color=uuuuuu] (0-0.05,1+0.05) node {$y$};
\draw[color=uuuuuu] (1+0.05,0-0.05) node {$x$};
\draw [fill=uuuuuu] (0,1-0.05) circle (0.1pt);
\draw[color=uuuuuu] (0-0.05,1-0.05) node {$6\lambda_t$};

\draw [fill=uuuuuu] (1/2,1-0.05) circle (0.1pt);
\draw [fill=uuuuuu] (1/2,0) circle (0.1pt);
\draw[color=uuuuuu] (1/2,0-0.05) node {$\frac{n_1-1}{n_1}$};

\draw [fill=uuuuuu] (2/3,1-0.05) circle (0.1pt);
\draw [fill=uuuuuu] (2/3,0) circle (0.1pt);
\draw[color=uuuuuu] (2/3,0-0.05) node {$\frac{n_2-1}{n_2}$};

\draw [fill=uuuuuu] (3/4,1-0.05) circle (0.1pt);
\draw [fill=uuuuuu] (3/4,0) circle (0.1pt);
\draw[color=uuuuuu] (3/4,0-0.05) node {$\cdots$};
\draw [fill=uuuuuu] (4/5,1-0.05) circle (0.1pt);
\draw [fill=uuuuuu] (4/5,0) circle (0.1pt);
%
\draw [fill=uuuuuu] (5/6,1-0.05) circle (0.1pt);
\draw [fill=uuuuuu] (5/6,0) circle (0.1pt);
\draw[color=uuuuuu] (5/6,0-0.05) node {$\frac{n_t-1}{n_t}$};

\draw [fill=uuuuuu] (1,1) circle (0.1pt);
\draw[color=uuuuuu] (1+0.04,1+0.04) node {$(1,1)$};

\draw [fill=uuuuuu] (1,4/9) circle (0.1pt);
\end{scriptsize}
\end{tikzpicture}
\caption{The function $g(\cM_t)$ has exactly $t$ global maxima.}
\end{figure}

Roughly speaking, the connection between these results is as follows. 
An $r$-graph is a {\em star} if there is a vertex $v$ such that all edges contain $v$,
and an $r$-graph~$\cH$ is {\em semibipartite} if it is $\cS$-colorable for some star $\cS$. 
Note that this is the same as saying that $V(\cH)$ has a partition into two parts $A$ and $B$ such 
that all edges have exactly one vertex in $A$ and $r-1$ vertices in $B$.
We will see later that our definition of $\cM_t$ ensures that every semibipartite $3$-graph 
is $\cM_t$-free.
By shrinking $A$, the shadow density of an $n$-vertex semibipartite 3-graph $\cH$ can be made 
arbitrarily close to $1$ as $n \rightarrow \infty$, so ${\rm proj}\Omega(\cM_t) = [0,1]$.
The shadows of the triple systems $\cG_1, \ldots, \cG_t$ from Theorem~\ref{THM:t-stability} 
are complete graphs and thus their edge densities are the distinct numbers $1-1/n_1, \dots, 1-1/n_t$. 
So $g(\cM_t, x)=6\lambda_t$ holds if $x$ is one of those densities and stability allows us to 
exclude further solutions to this equation. 

\subsection*{Organization} 
In Section~\ref{SEC:lagrangian} we present some definitions related to the Lagrangian of hypergraphs
and prove a result about the Lagrangian of a class of almost complete $3$-graphs.
In Section~\ref{SEC:the-extremal-configurations} we use the result from Section~\ref{SEC:lagrangian}
to define the extremal configurations,
which are balanced blow-ups of $\cG_1,\ldots, \cG_t$,
define the forbidden family $\cM_t$, and
prove the first part of Theorem~\ref{THM:t-stability}.
We prove the second part of Theorem~\ref{THM:t-stability} in Section~\ref{SEC:stability}, 
and Theorem~\ref{THM:feasible-region-Mt} in Section~\ref{subsection:feasible}. 
Section~\ref{SEC:remarks} contains some concluding remarks on generalisations to $r$-graphs
and open problems.

\section{Lagrangian}\label{SEC:lagrangian}

In this section we present some definitions related to the Lagrangian of a hypergraph,
introduced by Frankl and R\"odl in~\cite{FR84},
and prove a result (Proposition~\ref{prop:lagrange} below) about certain almost complete triple 
systems.

Let $\cG$ be an $r$-graph for some $r\ge 2$. The \emph{neighborhood} 
of a vertex $v\in V(\cG)$ is defined to be
\begin{align}
N_{\cG}(v) 
= 
\left\{u\in V(\cG)\setminus\{v\}\colon
\text{ there is $A\in \cG$ such that } \{u,v\}\subseteq A\right\}, \notag
\end{align}
the link of $v$ is
\[
	L_{\cG}(v) = \left\{A\in \partial\cG\colon A\cup\{v\}\in \cG\right\},
\]
and $d_{\cG}(v)= |L_{\cG}(v)|$ is called the \emph{degree} of $v$.
Denote by $\delta(\cG), \Delta(\cG)$ the minimum and maximum degree of $\cG$, respectively.
For a pair of vertices $u,v\in V(\cG)$
the neighborhood of $\{u,v\}$ is
\begin{align}
N_{\cG}(u,v) 
= 
\left\{w\in V(\cG)\setminus\{u,v\}\colon
                        \exists A\in \cG \text{ such that } \{u,v,w\}\subseteq A\right\}, \notag
\end{align}
and $d_{\cG}(u,v) = |N_{\cH}(u,v)|$ is called the \emph{codegree} of $\{u,v\}$.
Denote by $\delta_2(\cG), \Delta_2(\cG)$ the minimum and maximum codegree of $\cG$, respectively.

For an $r$-graph $\cG$ on $n$ vertices (let us assume for notational transparency 
that $V(\cG) = [n]$) the multilinear function
$L_{\cG}\colon \RR^n \to \RR$ is defined by
\begin{align}
L_{\cG}(x_1,\ldots,x_n) = \sum_{E\in \cH}\prod_{i\in E}x_i
\quad \text{ for all } (x_1,\ldots,x_n) \in \RR^n. \notag
\end{align}
Denote by $\Delta_{n-1}$ the standard $(n-1)$-dimensional simplex,
i.e.
\begin{align}
\Delta_{n-1} = \left\{(x_1,\ldots,x_n)\in [0,1]^n\colon x_1+\cdots+x_n= 1\right\}. \notag
\end{align}
Since $\Delta_{n-1}$ is compact,
a theorem of Weierstra\ss\ implies that the restriction of $L_{\cG}$ to $\Delta_{n-1}$
attains a maximum value, called the \emph{Lagrangian} of $\cG$ and denoted by $\lambda(\cG)$.

For a hypergraph $\cG$ the maximum number of edges in a blow-up of $\cG$
is related to $\lambda(\cG)$ (e.g. see Frankl and F\"uredi~\cite{FF89} 
or Keevash's survey~\cite{KE11}*{Section~3}).

\begin{lemma}[\cites{FF89, KE11}]\label{LEMMA:blowup-lagrangian}
Let $r \ge 2$ and let $\cG$, $\cH$ be two $r$-graphs.
If $\cH$ is a blow up of $\cG$,
then $|\cH| \le \lambda(\cG) v(\cH)^{r}$.
\end{lemma}

Given a $3$-graph $\cG$,
by plugging $(1/n, \ldots, 1/n)$ into $L_{\cG}$ one immediately obtains the lower bound
$\lambda(L_{\cG})\ge |\cG|/n^3$.
It is well known that for cliques~$\cH=K_n^{3}$ this holds with
equality and, moreover, that $(1/n, \ldots, 1/n)$ is the only point in the simplex~$\Delta_{n-1}$,
where $L_{\cH}$ attains this maximum value.

The main result of this section, Proposition~\ref{prop:lagrange} below, exhibits a class
of almost complete $3$-graphs having the same properties.
This will allow us later to construct for every given positive integer $t$ a 
family $\{\cG_1, \ldots, \cG_t\}$ of $3$-graphs
and a rational number~$\lambda_t$ close to $1/6$ such 
that $\lambda(\cG_i)=|\cG_i|/v(\cG_i)^3=\lambda_t$ holds for all $i\in [t]$.
The extremal configurations for our hypergraph Tur\'an problem are
then going to be balanced blow-ups of $\cG_1, \ldots, \cG_t$.
As we can accomplish $v(\cG_1)<\dots<v(\cG_t)$, this is relevant to 
Theorem~\ref{THM:feasible-region-Mt}.

Let us observe that every hypergraph $\cG$ satisfying $\lambda(\cG) = |\cG|/v(\cG)^3$ needs to be regular
in the sense that all vertices have the same degree.
In the converse direction, regular hypergraphs can still have much larger Lagrangians than $|\cG|/v(\cG)^3$.
For instance, the Lagrangian of the Fano plane is $1/27$ but not $1/49$.
To avoid such situations we utilize a design theoretic construction.

For the purposes of this article, by an \emph{$(n, k)$-design} we shall mean a $k$-graph
$\cD$ on $n$ vertices such that every pair of vertices is covered
by a unique edge.
With every such design $\cD$ we associate the $3$-graph 
\begin{align}
H(\cD) = \bigcup_{E\in \cD}\binom{E}{3}. \notag
\end{align}
on $V(\cD)$. Note that
\begin{align}
|H(\cD)| = \binom{k}{3}\frac{\binom{n}{2}}{\binom{k}{2}} = \frac{k-2}{6}n(n-1). \notag
\end{align}

It will turn out that for $n\ge 18k$ every $3$-graph of the form $\cG = K_{n}^{3}\setminus H(\cD)$,
where $\cD$ is an $(n,k)$-design on $[n]$, has the property $\lambda(\cG) = |\cG|/v(\cG)^3$.
In order to increase our control over the resulting value of $\lambda(\cG)$
Proposition~\ref{prop:lagrange}
allows the extra flexibility to subtract a very sparse regular $3$-graph from $\cG$.
Moreover, for reasons related to stability we state slightly more than just the actual value of the Lagrangian.

\begin{proposition}\label{prop:lagrange}
Suppose that $n\ge 18k+3^7s^3$, $\cD$ is an $(n, k)$-design on $[n]$,
and $\cS$ is an $s$-regular $3$-graph on $[n]$.
If $\cS \cap H(\cD) = \emptyset$ and $\cG = K_{n}^{3}\setminus\left(H(\cD) \cup \cS\right)$,
then
\begin{align}\label{eq:1533}
L_{\cG}(x_1, \ldots, x_n)+\frac19\sum_{i=1}^n\left(x_i-\frac1n\right)^2
\le \frac{|\cG|}{n^3}
= \frac{1}{6}\left(1-\frac{k+1}{n}+\frac{k-2s}{n^2}\right)
\end{align}
holds for all $(x_1, \dots, x_n)\in\Delta_{n-1}$ and, consequently, 
\begin{equation}\label{eq:1534}
\lambda(\cG)=\frac{1}{6}\left(1-\frac{k+1}{n}+\frac{k-2s}{n^2}\right).
\end{equation}
\end{proposition}

We start with a simple observation that will come in handily later.

\begin{fact}\label{f:1651}
Let $\cG$ be a $3$-graph with vertex set $[n]$ and let $\alpha\ge 0$ be a real number.
If the real numbers $\alpha_1, \dots, \alpha_n\in [-1, \alpha]$ sum up to zero, then
\begin{align}
L_{\cG}(\alpha_1, \ldots, \alpha_n) \le (\alpha n)^3. \notag
\end{align}
\end{fact}
\begin{proof}[Proof of Fact~\ref{f:1651}]
Define $P=\{i\in [n]\colon \alpha_i>0\}$ to be the set of vertices of $\cG$ with positive weight.
Let us decompose $L_{\cG}(\alpha_1, \ldots, \alpha_n)=S_0+S_1+S_2+S_3$ such that
for $m\in\{0, 1, 2, 3\}$ the sum $S_m$ consists of all terms $\alpha_i\alpha_j\alpha_k$
contributing to $L_{\cG}$ and satisfying $|P\cap \{i, j, k\}|=m$.
	
As the sums~$S_0$ and~$S_2$ possess no positive terms, we have~$S_0, S_2\le 0$.
Moreover,~$S_3$ has no more than $\binom{|P|}{3}\le n^3/6$ summands each of which
amounts to at most~$\alpha^3$, wherefore~$S_3$ is at most $(\alpha n)^3/6$.
Thus to conclude the argument it is more than enough to show~$S_1\le (\alpha n)^3/2$.
	
Writing $W=\sum_{i\in P}\alpha_i$ we have $\sum_{i\in [n]\setminus P}\alpha_i =-W$
and
\begin{align}
S_1
\le \sum_{i\in P}\alpha_i \cdot \sum_{jk\in \binom{[n]\setminus P}{2}}\alpha_j\alpha_k
\le W\cdot (W^2/2)
= W^3/2, \notag
\end{align}
which by $|W|\le \alpha |P|\le \alpha n$ completes the proof.
\end{proof}

\begin{proof}[Proof of Proposition~\ref{prop:lagrange}]
Since the left side of~\eqref{eq:1533} is continuous in $(x_1, \ldots, x_n)$ and~$\Delta_{n-1}$
is compact, there exists a point $\xi=(\xi_1, \ldots, \xi_n)\in \Delta_{n-1}$ such that
\begin{align}\label{eq:1830}
\omega
= L_{\cG}(\xi_1, \ldots, \xi_n)+\frac19\sum_{i=1}^n\left(\xi_i-\frac1n\right)^2
  - \frac{1}{6}\left(1-\frac{k+1}{n}+\frac{k-2s}{n^2}\right)
\end{align}
is maximum.
Assume for the sake of contradiction that $\omega>0$.
	
\begin{claim}\label{clm:1655}
There exists an index $i(\star)\in [n]$ such that $\xi_{i(\star)}>\frac 1n+\frac{9s}{n^2}$.
\end{claim}
\begin{proof}[Proof of Claim~\ref{clm:1655}]
Define $\alpha_1, \ldots, \alpha_n\in [-1, n-1]$ by $\xi_i=(1+\alpha_i)/n$ for every~$i\in [n]$ and observe that 		
\begin{align*}
\omega n^3
&= L_{\cG}(1+\alpha_1, \dots, 1+\alpha_n)+\frac n9\sum_{i=1}^n \alpha^2_i - |\cG| \\
&= \sum_{i=1}^n d_{\cG}(i)\alpha_i + \sum_{1\le i < j\le n}d_{\cG}(i, j)\alpha_i\alpha_j
	+ L_{\cG}(\alpha_1, \ldots, \alpha_n)+\frac n9\sum_{i=1}^n \alpha_i^2 \,.
\end{align*}
Since all vertices of $\cG$ have the same degree
and $\sum_{i=1}^n \alpha_i=n\bigl(\sum_{i=1}^n \xi_i-1\bigr)=0$,
the first sum on the right side vanishes.
Moreover, all pairs of vertices have codegree $n-k$ in~$K_{n}^3\setminus H(\cD)$ and thus we obtain
\begin{align}\label{eq:1628}
\omega n^3
= \left(\frac n9-\frac{n-k}2\right)\sum_{i=1}^n \alpha_i^2
 - \sum_{1\le i < j\le n}d_{\cS}(i, j)\alpha_i\alpha_j
 + L_{\cG}(\alpha_1, \ldots, \alpha_n)\,.
\end{align}

\medskip

{\it \hskip 1em First case: We have $\xi_1, \ldots, \xi_n>0$.}

\smallskip

Collecting the quadratic and cubic terms in~\eqref{eq:1628} separately we put
\begin{align}
Q
= \left(\frac n9-\frac{n-k}2\right)\sum_{i=1}^n \alpha_i^2
  - \sum_{1\le i < j\le n}d_{\cS}(i, j)\alpha_i\alpha_j
\quad {\rm and} \quad
K
= L_{\cG}(\alpha_1, \ldots, \alpha_n), \notag
\end{align}
so that
\begin{align}
\omega n^3 = Q+K. \notag
\end{align}
Now for every real number $C$
sufficiently close to $1$ the point $(\xi'_1, \ldots, \xi'_n)$ defined
by $\xi'_i=(1+C\alpha_i)/n$ belongs to~$\Delta_{n-1}$ and the maximal choice of $\omega$
reveals
\begin{align}
L_{\cG}(\xi'_1, \ldots, \xi'_n)
+ \frac19\sum_{i=1}^n\left(\xi'_i-\frac1n\right)^2
- \frac{1}{6}\left(1-\frac{k+1}{n}+\frac{k-2s}{n^2}\right)
\le \omega. \notag
\end{align}
Multiplying by $n^3$ and repeating the above calculation we obtain $QC^2 + KC^3 \le Q+K$ and thus
\begin{align}\label{eq:1613}
0\le (1-C)[(1+C)Q+(1+C+C^2)K]
\end{align}
whenever $|C-1|$ is sufficiently small.
Letting $C$ tend to $1$ from above and below we obtain $2Q+3K=0$.
Substituting this back into~\eqref{eq:1613} we learn
\begin{align}
0\le (1-C)[(C-1)Q+(C^2+C-2)K]= -(1-C)^2[Q+(C+2)K]. \notag
\end{align}
Thus $Q+3K\le (1-C)K$ holds whenever $|C-1|$ is sufficiently small,
which is only possible if $Q+3K\le 0$. Together with $Q+K=\omega n^3>0$
this yields $K<0$ and $\omega n^3<Q-K=(-1)^2Q+(-1)^3Q$. So the maximality 
of $\omega$ tells us that for $C=-1$ we have $(\xi'_1, \dots, \xi'_n)\not\in\Delta_{n-1}$.
In other words, there is some $i(\star)\in [n]$ such that 
\[
	\xi_{i(\star)}
	\ge 
	\frac 2n
	>
	\frac 1n+\frac{9s}{n^3},
\]
as desired.  
  
\medskip

{\it \hskip 1em Second case: There exists some $j(\star)\in [n]$ satisfying $\xi_{j(\star)}=0$.}

\smallskip
		
Now $\alpha_{j(\star)}=-1$ and, consequently,
\begin{align}\label{eq:1645}
\sum_{i=1}^n\alpha^2_i\ge 1.
\end{align}
		
Next we observe that the hypothesis that $\cS$ be $s$-regular yields
\begin{align}
-\sum_{1\le i < j\le n}d_{\cS}(i, j)\alpha_i\alpha_j
\le \sum_{1\le i < j\le n}d_{\cS}(i, j)\frac{\alpha_i^2+\alpha_j^2}2
= s\sum_{i=1}^n\alpha^2_i. \notag
\end{align}
Combined with~\eqref{eq:1628} and the positivity of $\omega$ this shows
\begin{align}\label{eq:1649}
\left(\frac{n-k}2-\frac n9-s\right)\sum_{i=1}^n\alpha^2_i
< L_{\cG}(\alpha_1, \ldots, \alpha_n).
\end{align}
Due to $n\ge 18k+3^7s^3$ we have
\begin{align}
\frac{n-k}2-\frac n9-s
> \left(\frac12-\frac 1{36}-\frac 19-\frac 1{36} \right)n
= \frac n3
> (9s)^3 \notag
\end{align}
and together with~\eqref{eq:1645},~\eqref{eq:1649} this establishes
\begin{align}
(9s)^3 < L_{\cG}(\alpha_1, \ldots, \alpha_n). \notag
\end{align}
In view of Fact~\ref{f:1651} we deduce that $\alpha_{i(\star)}>9s/n$ holds for some $i(\star)\in [n]$ and now
\begin{align}
\xi_{i(\star)}
= \frac{1+\alpha_{i(\star)}}{n}
> \frac 1n+\frac{9s}{n^2} \notag
\end{align}
follows. Thereby Claim~\ref{clm:1655} is proved.
\end{proof}
	
Now for every $i\in [n]$ we set
\begin{align}
D_i
= \frac{\partial L_{\cG}(x_1, \ldots, x_n)}{\partial x_i}\bigg|_{(\xi_1, \ldots, \xi_n)}
= \sum_{jk\in L_i}\xi_j\xi_k, \notag
\end{align}
where $L_i$ denotes the link graph of $i$ in $\cG$.
Owing to the maximality of $\omega$ in~\eqref{eq:1830} the Lagrange multiplier method leads
to the existence of a real number $M$ such that	
\begin{align}
D_i+\frac29\left(\xi_i-\frac1n\right)=M \notag
\end{align}
holds for every vertex $i\in [n]$ with $\xi_i>0$. Notice that
\begin{align*}
M
&= M\sum_{j=1}^n\xi_j
 = \sum_{j=1}^n\xi_j\left(D_j+\frac29\left(\xi_j-\frac1n\right)\right)
 = 3L_{\cG}(\xi_1, \ldots, \xi_n)+\frac 29\sum_{j=1}^n\left(\xi_j-\frac1n\right)^2 \\
&\overset{\eqref{eq:1830}}{>}
   \frac{1}{2}\left(1-\frac{k+1}{n}+\frac{k-2s}{n^2}\right)-\frac 19\sum_{j=1}^n\left(\xi_j-\frac1n\right)^2.
\end{align*}
Altogether, this proves
\begin{align}
D_i+\frac29\left(\xi_i-\frac 1n\right)+\frac19\sum_{j=1}^n\left(\xi_j-\frac1n\right)^2
> \frac{1}{2}\left(1-\frac{k+1}{n}+\frac{k-2s}{n^2}\right) \notag
\end{align}
for every vertex $i\in [n]$ satisfying $\xi_i>0$.
	
By our design theoretic construction, the link in $K_{n}^{3}\setminus H(\cD)$ of every vertex $i\in [n]$
is a $q$-partite Tur\'an graph with vertex classes of size $k-1$, where $q=(n-1)/(k-1)$
is an integer. Consequently, there exist real numbers $\beta_1, \ldots, \beta_q$ such
that $\xi_i+(\beta_1+\dots+\beta_q)=1$ and
\begin{align}
D_i
\le \sum_{1\le v<w\le q}\beta_v\beta_w
\le \frac{q-1}{2q}(\beta_1+\dots+\beta_q)^2
=   \frac{n-k}{2(n-1)}(1-\xi_i)^2. \notag
\end{align}
Summarizing, we have
\begin{align}\label{eq:1853}
\frac29\left(\xi_i-\frac 1n\right)+\frac19\sum_{j=1}^n\left(\xi_j-\frac1n\right)^2
> \frac{n-k}{2(n-1)}\left(\left(1-\frac 1n\right)^2-(1-\xi_i)^2\right)-\frac{s}{n^2}
\end{align}
for every vertex of positive weight. For the rest of the argument we fix a
vertex $i(\star)\in [n]$ such that $\xi_{i(\star)}$ is maximal.
Let us add the trivial estimate
\begin{align}
\frac19\sum_{j=1}^n \xi_j(\xi_{i(\star)}-\xi_j)\ge 0 \notag
\end{align}
to the case $i=i(\star)$ of~\eqref{eq:1853}. Because of 
\begin{align}
\sum_{j=1}^n\left(\xi_j-\frac1n\right)^2+\sum_{j=1}^n \xi_j(\xi_{i(\star)}-\xi_j)
&=
\sum_{j=1}^n\xi_j\left(\xi_{i(\star)}-\frac1n\right)-\frac1n\sum_{j=1}^n\left(\xi_j-\frac1n\right) \\
&=
\xi_{i(\star)}-\frac1n
\end{align}
this yields
\begin{align*}
\frac13\left(\xi_{i(\star)}-\frac 1n\right)
&>   \frac{n-k}{2(n-1)}\left(\xi_{i(\star)}-\frac 1n\right)\left(2-\frac 1n-\xi_{i(\star)}\right) -\frac s{n^2} \\
&\ge \frac{n-k}{2n}\left(\xi_{i(\star)}-\frac 1n\right)-\frac s{n^2}
 \ge \frac49\left(\xi_{i(\star)}- \frac 1n\right)-\frac s{n^2}, \notag
\end{align*}
whence
\begin{align}
\xi_{i(\star)} < \frac1n+\frac{9s}{n^2}. \notag
\end{align}
Owing to the maximal choice of $\xi_{i(\star)}$ this contradicts Claim~\ref{clm:1655}.
\end{proof}

%
%
\section{Constructions and Tur\'an numbers}\label{SEC:the-extremal-configurations}
Given a positive integer $t$ we define in this section the triple systems $\cG_1,\ldots, \cG_t$ 
and the forbidden family $\cM_t$ appearing in Theorem~\ref{THM:t-stability}. For every $i\in [t]$
there will be three integers $n_i$, $k_i$, $s_i$ such 
that $\cG_i=K^3_{n_i}\setminus (H(\cD_i)\cup \cS_i)$ holds for some $(n_i, k_i)$-design $\cD_i$
on~$[n_i]$ and some $s_i$-regular triple system $\cS_i$ on $[n_i]$ that is disjoint to $H(\cD_i)$.
As we shall have $n_i\gg k_i, s_i$, Proposition~\ref{prop:lagrange} will imply
\begin{align}
\lambda(\cG_i) = \frac{1}{6}\left(1-\frac{k_i+1}{n_i}+\frac{k_i-2s_i}{n_i^2}\right). \notag
\end{align}
Part of our goal is that balanced blow-ups of $\cG_1,\ldots, \cG_t$ should be extremal $\cM_t$-free 
triple systems and for this reason we need to ensure $\lambda(\cG_i) = \cdots = \lambda(\cG_t)$.
We shall achieve this by letting $k_i = 2s_i$ for $i\in[t]$,
and by guaranteeing  

\begin{align}\label{eq:3102}
\frac{k_1+1}{n_1} = \cdots = \frac{k_t+1}{n_t}. 
\end{align}

The details of this construction are given in Subsection~\ref{subsec:31} and the 
exact Tur\'an numbers of our families $\cM_t$ are determined in Subsection~\ref{subsec:32}.

\subsection{The extremal configurations and forbidden family}\label{subsec:31}

First, we need the following theorem about the existence of designs due to Wilson~\cites{W1,W2,W3}.

\begin{theorem}[Wilson~\cites{W1,W2,W3}]\label{THM:Wilson-design}
For every integer $k\ge 2$ there exists a threshold $n_0(k)$ such that for every
integer $n\ge n_0(k)$ satisfying the divisibility conditions $(k-1)\mid (n-1)$
and $(k-1)k \mid (n-1)n$ there exists an $(n, k)$-design.
\end{theorem}

Our next lemma deals with the arithmetic properties the numbers $k_1,\ldots, k_t$ 
and~$n_1,\ldots, n_t$ entering the construction of $\cG_1, \ldots, \cG_t$ need to satisfy. 
Apart from~\eqref{eq:3102} and the divisibility conditions in
Theorem~\ref{THM:Wilson-design} we will require that $n_1, \ldots, n_t$ are divisible by $3$
so that $(k_i/2)$-regular triple systems on $n_i$ vertices exist. Thus the case $q=3$ of the
following lemma is exactly what we need. 

\begin{lemma}\label{LEMMA:existence-extremal-configurations}
Given positive integers $t$ and $q$ there exist $t$ even integers $3 < k_1 < \cdots < k_t$ 
such that for every constant $C > 0$ there exist $t$ integers $n_1 < \cdots < n_t$ with 
the following properties. 
\begin{enumerate}[label=\alabel]
\item\label{it:32a} We have $q\mid n_i$, $(k_i-1) \mid (n_i-1)$, and $k_i(k_i-1) \mid n_i(n_i-1)$ 
	for all $i \in [t]$. 
\item\label{it:32b}	Moreover,
\begin{align}
Q=\frac{n_1}{k_1+1} = \dots = \frac{n_t}{k_t+1} \notag
\end{align}
is an integer with $Q\ge C$.
\end{enumerate}
\end{lemma}

\begin{proof}[Proof of Lemma \ref{LEMMA:existence-extremal-configurations}]
Starting with an arbitrary positive multiple $s_1$ of $q$ we recursively define 
integers $1\le s_1 < \cdots < s_t$ by setting $s_{i+1} = \prod_{j\le i}s_j(2s_j-1)+1$ 
for every $i\in [t-1]$. Now whenever $1\le i<j\le t$ we have $s_j\equiv 1\pmod{s_i(2s_i-1)}$
and, consequently, 
\[
s_j(2s_j-1)\equiv 1\pmod{s_i(2s_i-1)}. 
\]
In particular, the numbers 
\[
s_1(2s_1-1),\ldots, s_t(2s_t-1)
\]
are pairwise coprime and by the Chinese remainder theorem there exists 
an even integer $Q\ge C$ such that $Q/2 \equiv s_{i}^2 \pmod {s_{i}(2s_i-1)}$ holds for 
all $i\in [t]$. Multiplying these congruences by $2$ and setting $k_i=2s_i$ we obtain 
\begin{equation}\label{eq:chooseQ}
Q\equiv k_i^2/2\pmod{k_i(k_i-1)}.
\end{equation} 
Now it is plain that the numbers $n_i=Q(k_i+1)$ satisfy~\ref{it:32b}. Moreover, the case $i=1$
of~\eqref{eq:chooseQ} yields $q\mid k_1\mid Q$ and, therefore, $n_1, \ldots, n_t$ are divisible 
by~$q$. Finally, multiplying~\eqref{eq:chooseQ} by $k_i+1$ we learn 
\[
n_i\equiv k_i(k_i+1)(k_i/2) \equiv 2k_i(k_i/2)\equiv k_i^2\equiv k_i \pmod {k_i(k_i-1)},
\]
for which reason $k_i\mid n_i$ and $(k_i-1)\mid (n_i-1)$. So altogether~\ref{it:32a} holds as well.   
\end{proof}

Given two $r$-graphs $\cH_1$ and $\cH_2$ with the same number of vertices
a {\it packing} of~$\cH_1$ and~$\cH_2$ is a bijection $\phi\colon V(\cH_1) \to V(\cH_2)$
such that $\phi(E)\not\in \cH_2$ for all $E\in \cH_1$. In order to proceed with our construction 
of the triple systems $\cG_1, \ldots, \cG_t$ we need to argue that, under natural assumptions, 
if $\cD_i$ denotes an $(n_i, k_i)$-design, then there is an $s_i$-regular $3$-graph 
$\cS_i\subseteq K^3_n\setminus H(\cD_i)$, where $s_i=k_i/2$. 
Provided that $3\mid n_i$ and $s_i\le \binom{n-1}2$ the existence of some $s_i$-regular 
$3$-graph $\cS_i\subseteq K^3_n$ is a well known fact that follows, e.g., from Baranyai's 
factorisation theorem~\cite{BA75}. For making $\cS_i$ and $H(\cD_i)$ disjoint we use a packing 
argument based on the following result of Lu and Sz\'{e}kely. 
 
\begin{theorem}[Lu-Sz\'{e}kely \cite{LS07}]\label{THEOREM:packing-hypergraphs}
Let $\cH_1$ and $\cH_2$ be two $r$-graphs on $n$ vertices.
If
\begin{align}
\Delta(\cH_1) |\cH_2| + \Delta(\cH_2) |\cH_1| < \frac{1}{er} \binom{n}{r}, \notag
\end{align}
then there is a packing of $\cH_1$ and $\cH_2$.
\end{theorem}

In fact, we only require the following consequence. 

\begin{corollary}\label{CORO:packing-S-HD}
Suppose $3\mid n$ and that $\cD$ is an $(n,k)$-design on $[n]$.
If $s < \frac{n-2}{6e(k-2)}$, then there exists an
$s$-regular $3$-graph $\cS$ on $[n]$ such that $\cS \cap H(\cD) = \emptyset$.
\end{corollary}

\begin{proof}[Proof of Corollary~\ref{CORO:packing-S-HD}]
By $3\mid n$ and $s\le\binom{n-1}2$ there is an $s$-regular $3$-graph $\cS'$ on $n$ vertices.
Since
\begin{align}
\Delta(\cS') |H(\cD)| + \Delta(H(\cD)) |\cS'|
& = s  \frac{k-2}{6}n(n-1) + \frac{k-2}{2}(n-1) \frac{sn}{3} \notag\\
& = s \frac{k-2}{3}n(n-1)
 <  \frac{n-2}{6e(k-2)}\frac{k-2}{3}n(n-1) \notag\\
& = \frac{1}{3e}\binom{n}{3}, \notag
\end{align}
Theorem~\ref{THEOREM:packing-hypergraphs} yields a packing $\phi\colon V(\cS') \to [n]$ 
of $\cS'$ and $H(\cD)$. It is clear that $\cS = \phi(\cS')$ 
satisfies the requirements of Corollary~\ref{CORO:packing-S-HD}.
\end{proof}

Now we are ready to present the definition of $\cG_1, \ldots, \cG_t$.

\begin{constr}\label{con:Gi}
Given a positive integer $t$ perform the following steps. 
\begin{itemize}
\item Apply Lemma~\ref{LEMMA:existence-extremal-configurations} with $q=3$, thus getting 
	some even integers $3< k_1 < \cdots < k_t$.
\item Take an integer $C\ge \max\{n_0(k_1),\ldots,n_{0}(k_t), 2k_t^3, 3^8\}$, where the thresholds 
	$n_0(k_i)$ are given by Theorem~\ref{THM:Wilson-design}.
\item Now Lemma~\ref{LEMMA:existence-extremal-configurations} applied to $C$ and $k_1, \dots, k_t$
	yields integers $C<n_1<\dots<n_k$ such that, in particular, 
	\[
		Q=\frac{n_1}{k_1+1}=\dots=\frac{n_t}{k_t+1}
	\]
	is an integer with $Q\ge C$. 
\end{itemize}
Now, for every $i\in [t]$
\begin{itemize}
\item let $\cD_i$ be an $(n_i,k_i)$-design on $[n_i]$ (as obtained by Theorem~\ref{THM:Wilson-design})
\item let $\cS_i$ be a $(k_i/2)$-regular $3$-graph on $[n_i]$
such that $\cS_i \cap H(\cD_i) = \emptyset$ (as obtained by Corollary~\ref{CORO:packing-S-HD}).
\item and, finally, define
\begin{align}
\cG_i = K_{n_i}^{3} \setminus \left(H(\cD_i) \cup \cS_i\right). \notag
\end{align}
\end{itemize}
\end{constr}

By Proposition~\ref{prop:lagrange} we have 
\begin{align}
\lambda(\cG_i)
=  \frac{1}{6}\left(1-\frac{k_i+1}{n_i}+ \frac{k_i-2k_i/2}{n_i^2}\right)
=  \frac{1}{6}\left(1-\frac{1}{Q}\right). \notag
\end{align}
for every $i\in [t]$, so some rational $\lambda_t$ satisfies 
\begin{equation}\label{eq:lambdat}
\lambda_t = \lambda(\cG_1) = \cdots = \lambda(\cG_t)\in [5/32, 1/6).
\end{equation}

In the remainder of this subsection we introduce the family $\cM_t$. 
For an $r$-graph $\cH$ and a set $S\subseteq V(\cH)$ we say that $S$ is \emph{$2$-covered} 
in $\cH$ if for every pair of vertices in $S$ there is an edge in $\cH$ containing it.
If this holds for $S=V(\cH)$ then $\cH$ itself is said to be $2$-covered.

For all integers $\ell>r\ge 2$ we let $\cK_{\ell}^{r}$ denote the family of $r$-graphs $F$ with 
at most~$\binom{\ell}{2}$ edges that contain a $2$-covered set $S$ of $\ell$ vertices called 
a~\emph{core} of $F$. 
The family $\cK_{\ell}^{r}$ was first introduced by the second 
author~\cite{MU06} in order to extend Tur\'{a}n's theorem to hypergraphs.
It also plays a key r\^{o}le in in the construction of the family $\cM$ with two 
extremal configurations in~\cite{LM2}. In the present work, we also need the larger family 
$\widehat{\cK}_{\ell}^{r}$ defined to consist of all $r$-graphs $F$ with at most $\binom{\ell}{r}$ 
edges that contain a $2$-covered set $S$ of $\ell$ vertices, which is again called a \emph{core}
of $F$. 

Let us recall that the \emph{transversal number} of a hypergraph $\cH$ is the nonnegative integer \begin{align}
\tau(\cH) =
\min\left\{ |S| \colon S \subseteq V(\cH) \text{ and } S\cap E \neq \emptyset \text{ for all } E \in \cH \right\}. \notag
\end{align}
Note that if $\cH$ is empty, then we can take $S=\varnothing$, whence $\tau(\cH) = 0$ holds in this 
case. After these preparations, the family $\cM_t$ is defined as follows. 

\begin{definition}\label{d:1151}
For every positive integer $t$ the family $\cM_t$
consists of all $3$-graphs $F\in\bigcup_{\ell\le n_t}\widehat{\cK}_{\ell}^{3}$
which do not occur as a subgraph in any blow-up of $\cG_1,\ldots,\cG_t$
and which have a core $S$ such that $\tau(F[S])\ge 2$.
\end{definition}

We conclude this subsection with a simple sufficient condition for $3$-graphs $F\in \cK^3_{n_t+1}$
guaranteeing that they are in $\cM_t$ (see Lemma~\ref{l:6521} below). For this purpose we 
require the following observation analysing the extent to which $\tau(\cH)\ge 2$ is a ``local'' 
property of a hypergraph $\cH$. 

\begin{fact}\label{f:1912}
	If $r\ge 2$ and $\cH$ denotes an $r$-graph with $\tau(\cH)\ge 2$, then there is a subgraph
	$\cH'\subseteq \cH$ with at most $r+1$ edges satisfying $\tau(\cH')\ge 2$.
\end{fact}

\begin{proof}
	Pick two distinct edges $E', E''\in \cH$ and write $E'\cap E''=\{v_1, \ldots, v_m\}$, where 
	$0\le m\le r-1$. For every $i\in [m]$ the assumption that $\{v_i\}$ fails to cover $\cH$
	yields an edge $E_i\in\cH$ such that $v_i\not\in E_i$. Now $\cH'=\{E', E'', E_1, \dots, E_m\}$
	has the desired properties. 
\end{proof}

Notice that the example $\cH=K^r_{r+1}$ shows that the bound $|\cH'|\le r+1$ is optimal. 

\begin{lemma}\label{LEMMA:subgraph-is-the-expansion-of-clique}
Suppose that $F$ is a $3$-graph and that $S\subseteq V(F)$ is a $2$-covered set in~$F$. 
If $\tau(F[S]) \ge 2$, then $F$ contains a subgraph $F'$ such that $F'\in\cK^3_{|S|}$
and $\tau(F'[S]) \ge 2$. Moreover, if $12\le s\le |S|$, then $F$ has a subgraph $F''\in\cK^3_s$
possessing a core $S''$ such that $\tau(F''[S'']) \ge 2$.
\end{lemma}

\begin{proof}[Proof of Lemma~\ref{LEMMA:subgraph-is-the-expansion-of-clique}]
The case $r=3$ of Fact~\ref{f:1912} yields a subgraph $\cG$ of $F[S]$ 
with at most four edges such that $\tau(\cG)\ge 2$. 
Notice that $|\cG| \ge 2$ and $|\partial \cG| \ge 5$.
Since $S$ is $2$-covered in~$F$, we can choose for every 
pair $uw\in \binom{S}{2}\setminus \partial \cG$
an edge $e_{uw}\in F$ containing $u$ and $w$.
Now
\begin{align}
F' = \left\{e_{uw}\colon uw\in \tbinom{S}{2}\setminus\partial \cG\right\} \cup \cG \notag
\end{align}
has the properties that $S$ is $2$-covered in $F'$ and $\tau(F'[S]) \ge 2$. Together with
\begin{align}
|F'| \le \binom{\ell}{2} - |\partial \cG| + |\cG|
\le \binom{\ell}{2} - 5 + 4
< \binom{\ell}{2} \notag
\end{align}
this proves $F'\in\cK_{|S|}^3$. Moreover, if any $s\in [12, |S|]$ is given, we can take 
a set $S''$ of size $s$ with $V(\cG)\subseteq S''\subseteq S$ and apply the first 
part of the lemma to $S''$ rather than $S$.  
\end{proof}

\begin{lemma}\label{l:6521}
	If $S$ denotes a core of $F\in\cK^3_{n_t+1}$ and $\tau(F[S])\ge 2$, then $F\in \cM_t$. 
\end{lemma}

\begin{proof}
By the previous lemma and $n_t\ge 12$ there exists a set $S''\subseteq S$ such 
that $|S''|=n_t$ and $\tau(F[S''])\ge 2$. Since $|F|\le \binom{n_t+1}2\le \binom{n_t}3$,
we can regard $F$ as a member of $\widehat{\cK}^3_{n_t}$ with core $S''$ and it remains 
to prove that $F$ cannot be $\cG_i$-colorable for any $i\in [t]$. This is due to the fact 
that the shadows of blow-ups of $\cG_i$ are complete $n_i$-partite graphs, while $S$ 
induces a $K_{n_t+1}$ in $\partial F$.  
\end{proof}

\subsection{Tur\'{a}n numbers of \texorpdfstring{$\cM_{t}$}{M}} \label{subsec:32}
 
Having now introduced the main protagonists $\cG_1, \ldots, \cG_t$ and $\cM_t$ we shall determine 
the extremal numbers $\ex(n, \cM_t)$ in this subsection. More precisely, setting
\begin{align}
\mathfrak{M}(n)
= 
\max\left\{|\cG|\colon \cG \text{ is } \cG_i\text{-colorable for some $i\in[t]$ and } 
v(\cG) = n\right\} \notag
\end{align}
for every positive integer $n$ we shall prove the following result. 

\begin{theorem}\label{THEOREM:turan-density-Mt}
The equality ${\rm ex}(n,\cM_{t}) = \mathfrak{M}(n)$ holds for all positive integers $n$.
\end{theorem}

Notice that in view of Lemma~\ref{LEMMA:blowup-lagrangian} and~\eqref{eq:lambdat} this implies 
$\ex(n, \cM_t) \le \lambda_t n^3$ for every positive integer $n$. Moreover, whenever $n$ is 
a multiple of $n_i$ for some $i\in[t]$, the balanced blow-up of $\cG_i$ with factor $n/n_i$
exemplifies that this holds with equality. For these reasons, Theorem~\ref{THEOREM:turan-density-Mt}
is stronger than Theorem~\ref{THM:t-stability}~\ref{it:11a}. Let us start with the lower bound 
on~${\rm ex}(n,\cM_t)$.

\begin{fact}\label{PROP:lower-bound-Turan-number}
We have ${\rm ex}(n,\cM_t) \ge \mathfrak{M}(n)$ for every positive integer $n$. 
\end{fact}

\begin{proof}[Proof of~\ref{PROP:lower-bound-Turan-number}]
This is an immediate consequence of the fact that by Definition~\ref{d:1151}
for every $i\in [t]$ all blow-ups of $\cG_i$ are $\cM_t$-free. 
\end{proof}

Our proof for the upper bound uses the Zykov symmetrization method~\cite{Zy}. 
The applicability of this technique
in the current situation hinges on the fact that if a hypergraph $\cH$ is $\cM_t$-free, then there
is no homomorphism from a member of $\cM_t$ to $\cH$ (see Proposition~\ref{PROP:Mt-homo-free} below).  
Let us recall that given two $r$-graphs $F$ and $\cH$ a map $\phi\colon V(F) \lra V(\cH)$ is said 
to be a \emph{homomorphism} if $\phi$ preserves edges, i.e., if $\phi(E) \in \cH$ holds for 
all $E \in F$. 
Further, $\cH$ is \emph{$F$-hom-free} if there is no homomorphism from $F$ to $\cH$
or, in other words, if~$F$ fails to be $\cH$-colourable. For a family $\cF$ of $r$-graphs, 
we say that $\cH$ is \emph{$\cF$-hom-free} if it is $F$-hom-free for every $F\in \cF$.

\begin{proposition}\label{PROP:Mt-homo-free}
A $3$-graph $\cH$ is $\cM_{t}$-hom-free if and only if it is $\cM_t$-free.
\end{proposition}

\begin{proof}[Proof of Proposition \ref{PROP:Mt-homo-free}]

Notice that the forward implication is clear. Now suppose conversely that $\cH$ fails to be 
$\cM_{t}$-hom-free, i.e., that there is a homomorphism $\phi\colon V(F)\lra V(\cH)$
for some $F\in\cM_t$. Clearly the restriction of $\phi$ to a core $S$
of $F$ is injective. So $\phi(F)\in\widehat{\cK}^3_{|S|}\cap \cM_t$ and 
in view of $\phi(F)\subseteq \cH$ it follows that $\cH$ fails to be $\cM_t$-free.
\end{proof}

As an immediate consequence of Definition~\ref{d:1151}, semibipartite triple systems 
are $\cM_t$-free. We analyze the semibipartite case as follows. 

\begin{lemma}\label{l:1237}
If $\cH$ denotes a semibipartite triple system on $n$ vertices, then 
\[
	|\cH|\le \min\{2n^3/27, \gM(n)\}.
\]
\end{lemma}

\begin{proof}
	Fix a partition $V(\cH)=A\dcup B$ such that $|E\cap A|=1$ holds for every $E\in \cH$.
	Now the AM-GM inequality yields 
	\[
		|\cH|
		\le 
		|A|\binom{|B|}2
		\le 
		\frac{2|A|\cdot |B|\cdot |B|}4
		\le 
		\frac14\left(\frac{2|A|+|B|+|B|}3\right)^3
		=
		\frac{2n^3}{27}
	\]
	and it remains to show $|\cH|\le \gM(n)$. If $n$ is large this is an immediate consequence 
	of $\gM(n)=(\lambda_t-o(1))n^3$ and $\lambda_t\ge 5/32>2/27$, but for a complete proof 
	addressing all values of $n$ we need to argue more carefully. 
	
	To this end we consider a random map $\phi\colon [n]\lra [n_1]$ together with the random 
	blow-up~$\widehat{\cG}$ of~$\cG_1$ determined by $\phi$. Explicitly $\widehat{\cG}$ has vertex
	set $[n]$ and a triple $ijk$ forms an edge of $\widehat{\cG}$ if and only 
	if $\phi(i)\phi(j)\phi(k)\in \cG_1$. Now every potential edge of $\widehat{\cG}$ is present 
	with probability $\frac{6|\cG_i|}{n_1^3}=6\lambda_t$
	and thus the expectation of $|\widehat{\cG}|$ is $6\lambda_t\binom n3$. So by averaging 
	we obtain 
	\begin{equation}\label{eq:0015}
		\gM(n)\ge 6\lambda_t\binom n3 \ge \frac{15}{16}\binom n3,
	\end{equation}
	which for $n\ge 5$ implies the desired estimate $\gM(n)\ge 2n^3/27$. Moreover,~\eqref{eq:0015}
	yields $\gM(4)\ge 3$, which still suffices for the case $n=4$ of our lemma. Finally, the case 
	$n\le 3$ is trivial. 
\end{proof}

The central notion in arguments based on Zykov symmetrization is the following: 
Given an $r$-graph $\cH$, two non-adjacent vertices $u,v\in V(\cH)$
are said to be \emph{equivalent} if $L_{\cH}(u) = L_{\cH}(v)$.
Evidently, equivalence is an equivalence relation. Since any two equivalent 
vertices have the same degree and the same link, we can write $d_\cH(C)$ and $L_{\cH}(C)$
for the common degree and the common link of all vertices in an equivalence class~$C$,
respectively.   

\begin{lemma}\label{LEMMA:complete-shadow-imply-Gi-colorable}
Let $\cH$ be an $\cM_t$-free $3$-graph with equivalence classes $C_1, \ldots, C_m$.
If for all distinct $k,\ell \in [m]$ the shadow $\partial\cH$ induces a complete bipartite graph
between~$C_k$ and~$C_{\ell}$, then~$\cH$ is either semibipartite or $\cG_i$-colourable 
for some $i\in [t]$.
\end{lemma}

\begin{proof}[Proof of Lemma~\ref{LEMMA:complete-shadow-imply-Gi-colorable}]
Let $T\subseteq V(\cH)$ be a set containing exactly one vertex from each equivalence class of $\cH$,
and let $\cT$ be the subgraph of $\cH$ induced by $T$. By assumption, $\cT$ 
is $2$-covered, $|T| = m$, and~$\cH$ is a blow-up of $\cT$.
If $\tau(\cT) < 2$, then~$\cT$ is a star and $\cH$ is semibipartite. So we may assume $\tau(\cT)\ge 2$
from now on. 

Since $\cT$ is $2$-covered and $|\cT|\le \binom{m}{3}$ we have $\cT \in \widehat{\cK}_{m}^{3}$.
So if $m\le n_t$, then in view of Definition~\ref{d:1151} 
and $\cT\not\in\cM_{t}$ there exists an index $i\in [t]$ such that $\cT$ is $\cG_i$-colorable. 
As $\cH$ is a blow-up of $\cT$, it follows that $\cH$ is $\cG_i$-colorable as well. 
  
Now assume for the sake of contradiction that $m>n_t$. 
Since $n_t\ge 12$, Lemma~\ref{LEMMA:subgraph-is-the-expansion-of-clique}
leads to a subgraph $\cT''\in\cK^3_{n_t+1}$ of $\cT$ having a core $S''$ such that 
$\tau(\cT''[S''])\ge 2$. By Lemma~\ref{l:6521} this contradicts $\cH$ being $\cM_t$-free.
\end{proof}

Now we are ready to establish the main result of this subsection. 

\begin{proof}[Proof of Theorem~\ref{THEOREM:turan-density-Mt}]
Fix some positive integer $n$. By Fact~\ref{PROP:lower-bound-Turan-number}
it suffices to establish the upper bound $\ex(n, \cM_t)\le \gM(n)$. Arguing 
indirectly we choose an $\cM_t$-free triple system~$\cH$ on~$n$ vertices with 
more than $\gM(n)$ edges such that the number $m$ of equivalence classes of $\cH$ 
is minimal. Let $C_1, \ldots, C_m$ be the equivalence classes of~$\cH$. 

By Lemma~\ref{l:1237} we know that $\cH$ is not semibipartite and the definition 
of $\gM(n)$ implies that $\cH$ fails to be $\cG_i$-colorable for every $i\in [t]$.
For these reasons, Lemma~\ref{LEMMA:complete-shadow-imply-Gi-colorable} tells us 
that~$\partial H$ is not the complete $m$-partite graph with vertex classes $C_1, \dots, C_m$.
Without loss of generality we may assume that at least one possible edge between~$C_1$ and~$C_2$ is 
missing in~$\partial H$. Due to the definition of equivalence there are actually no edges 
between~$C_1$ and~$C_2$ in~$\partial H$. By symmetry we may suppose further 
that $d_\cH(C_1)\le d_\cH(C_2)$. 

Now let $\cH'$ be the unique $3$-graph satisfying $V(\cH')=V(\cH)$, $\cH'-C_1=\cH-C_1$,
and $L_{\cH'}(v)=L_\cH(w)$ for all $v\in C_1$ and $w\in C_2$. Observe that 
$\{C_1\cup C_2, C_3, \dots, C_m\}$ refines the partition of $V(\cH')$ into the 
equivalence classes of $\cH'$ and 
\[
|\cH'|=|\cH|+|C_1|\bigl(d_\cH(C_2)-d_\cH(C_1)\bigr)\ge |\cH| >\gM(n).
\]
So our minimal choice of $m$ implies that $\cH'$ cannot be $\cM_t$-free. As there exists 
a homomorphism from $\cH'$ to $\cH$, it follows that $\cH$ fails to be $\cM_{t}$-hom-free.
But owing to Proposition~\ref{PROP:Mt-homo-free} this contradicts~$\cH$ 
being $\cM_t$-free.
\end{proof}

\section{Stability}\label{SEC:stability}
In this section we prove most of Theorem~\ref{THM:t-stability}~\ref{it:11b} --
only the proof of $\xi(\cM_t)=t$ is postponed to Section~\ref{subsection:feasible}. 
Our goal is to show that after deleting a small number of low-degree vertices an ``almost 
extremal'' $\cM_t$-free 
$3$-graph becomes $\cG_i$-colorable for some $i\in[t]$. More precisely, we aim for the following
result.

\begin{theorem}\label{THM:H-Z-is-Gi-colorable}
If $\epsilon >0$ is sufficiently small, $n$ is sufficiently large, and $\cH$ is an $\cM_t$-free
$3$-graph on $n$ vertices with $|\cH| \ge (\lambda_t-\epsilon)n^3$,
then the set 
\begin{align}
Z = \left\{u\in V(\cH)\colon d_{\cH}(u) \le (3\lambda_t-2\epsilon^{1/2})n^2\right\} \notag
\end{align}
has size at most $\epsilon^{1/2}n$ and
the $3$-graph $\cH-Z$ is $\cG_i$-colorable for some $i\in [t]$.
\end{theorem}

As the proof of this result will occupy the entire section, we would like to start 
with a quick overview. The argument is somewhat similar in spirit to~\cites{PI08,BIJ17,LM2}
and ultimately it is based on the Zykov symmetrization method~\cite{Zy}. There are certain kinds of 
complications that often arise when one uses this strategy in order to establish stability 
results and we overcome several of these common difficulties by introducing the 
so-called $\Psi$-trick in Subsection~\ref{subsec:41}. By means of this trick, the problem 
to prove Theorem~\ref{THM:H-Z-is-Gi-colorable} gets reduced to an apparently much simpler 
task: If a triple system $\cH$ with $n$ vertices and minimum degree 
$(3\lambda_t-o(1)) n^2$ can be made $\cG_i$-colorable by deleting a single vertex, 
then, actually,~$\cH$ itself 
is $\cG_i$-colorable (see Lemma~\ref{LEMMA:extend-coloring}). The $\Psi$-trick can also be used 
to reprove some known stability results with improved control over the dependence of the constants
(see~\cite{LMR2}). 

The proof of Lemma~\ref{LEMMA:extend-coloring} is still quite long. We will collect some auxiliary 
results in Subsection~\ref{subsec:42} and defer the main part of the argument to 
Subsection~\ref{subsec:43}

\subsection{General preliminaries.}\label{subsec:41}
This subsection reduces the task of proving Theorem~\ref{THM:H-Z-is-Gi-colorable} 
to the apparently much simpler task of verifying Lemma~\ref{LEMMA:extend-coloring} below. 
There are only few ``special properties'' of $\cM_t$ we are going to utilize in the course of this 
reduction and we refer to~\cite{LMR2} for a more systematic treatment. 

Throughout this subsection we use the following notation: For every $3$-graph $\cH$ on $n$ 
vertices and every $\eps>0$ we set 
\[
Z_\eps(\cH) 
= 
\left\{u\in V(\cH)\colon d_{\cH}(u) \le (3\lambda_t-2\epsilon^{1/2})n^2\right\}.
\]

\begin{lemma}\label{LEMMA:size-Z-H}
If $\eps\in (0, 1)$, $n\ge \eps^{-1/2}$ and $\cH$ is an $\cM_t$-free $3$-graph on $n$ vertices 
with at least $(\lambda-\epsilon)n^3$ edges, then
\begin{enumerate}[label=\alabel]
\item\label{it:42a} the set $Z_\eps(\cH)$ has at most the size $\epsilon^{1/2} n$
\item\label{it:42b} and the subgraph $\cH'=\cH-Z_\eps(\cH)$ of $\cH$ 
            satisfies $\delta(\cH') \ge (3\lambda_t-3\epsilon^{1/2})n^2$ as well as 
            $|\cH'| \ge (\lambda_t-2\epsilon^{1/2})n^3$.
\end{enumerate}
\end{lemma}

\begin{proof}[Proof of Lemma~\ref{LEMMA:size-Z-H}]
Set $Z=Z_\eps(\cH)$. Assuming that part~\ref{it:42a} fails we can take a set $X\subseteq Z$ of 
size $\frac23\epsilon^{1/2}n \le |X| \le 2\epsilon^{1/2}n$. 
The definition of $Z$ leads to
\begin{align*}
|\cH-X|
& \ge (\lambda_t-\epsilon)n^3 - |X|(3\lambda_t-2\epsilon^{1/2})n^2 \\
& \ge (\lambda_t-\epsilon)n^3 - |X|(3\lambda_t-2\epsilon^{1/2})n^2
	-\tfrac34 n(|X|-\tfrac23\epsilon^{1/2}n)(2\epsilon^{1/2}n-|X|) \\
&  =  \lambda_t (n-|X|)^3+3(1/4-\lambda_t)n|X|^2+\lambda_t|X|^3
   >  \lambda_t (n-|X|)^3,
\end{align*}
where we used $\lambda_t<1/6<1/4$ in the last step.
However, by Theorem~\ref{THM:t-stability}~\ref{it:11a}
this contradicts the fact that $\cH-X$ is $\cM_t$-free.

Now we prove part~\ref{it:42b}. For every $u\in V(\cH')$ the definition 
of $Z$ and~\ref{it:42a} yield 
\begin{align}
d_{\cH'}(u) 
\ge d_{\cH}(u) - |Z|n
\ge (3\lambda_t-2\epsilon^{1/2})n^2 - \epsilon^{1/2} n^2
=   (3\lambda_t-3\epsilon^{1/2})n^2. \notag
\end{align}
Similarly, we have 
\begin{equation*}
|\cH'|
\ge |\cH| - |Z|n^2
\ge (\lambda_t-\epsilon)n^3 - \epsilon^{1/2} n^3
>   (\lambda_t-2\epsilon^{1/2})n^3. \qedhere 
\end{equation*}
\end{proof}

The following lemma will be shown to imply Theorem~\ref{THM:H-Z-is-Gi-colorable}.

\begin{lemma}
\label{LEMMA:extend-coloring}
There exist constants $\zeta \in (0, 1)$ and $N_0\in\NN$ such that the following holds for all $n\ge N_0$.
Let $\cH$ be an $\cM_t$-free $3$-graph on $n$ vertices with at least $(\lambda_t-\zeta)n^3$ edges
and $\delta(\cH) > (3\lambda_t-\zeta)n^2$.
If there exists a vertex $v \in V(\cH)$ such that $\cH-v$ is $\cG_{i}$-colorable for some $i \in [t]$,
then $\cH$ itself is $\cG_{i}$-colorable as well.
\end{lemma}

We postpone the proof of this result to Subsection~\ref{subsec:43}. The deduction of 
Theorem~\ref{THM:H-Z-is-Gi-colorable} from Lemma~\ref{LEMMA:extend-coloring} factorises 
through the following statement. 

\begin{lemma}\label{LEMMA:extend-coloring-equivalent-class}
There exists $\eps\in (0, 1/16)$ such that the following holds for every sufficiently large 
integer $n$.
Let $\cH$ denote an $\cM_t$-free $3$-graph with $n$ vertices and at least $(\lambda_t-\epsilon)n^3$ 
edges.
If $Q\subseteq V(\cH)$ has size $|Q|\le 2\epsilon^{1/2}n$ and $\cH-Q$ is $\cG_i$-colourable for some $i\in [t]$,
then $\cH-Z_\eps(\cH)$ is $\cG_i$-colourable as well.
\end{lemma}

\begin{proof}[Proof of Lemma~\ref{LEMMA:extend-coloring-equivalent-class} using Lemma~\ref{LEMMA:extend-coloring}]
We show that $\eps=\zeta^2/25$ has the desired property, where $\zeta$ denotes the constant 
provided by Lemma~\ref{LEMMA:extend-coloring}. Given a sufficiently large $3$-graph~$\cH$ 
and a set~$Q$ as described in the statement of Lemma~\ref{LEMMA:extend-coloring-equivalent-class}
we set $Q' = Q\setminus Z_\eps(\cH)$ and $V' = V(\cH)\setminus (Z_\eps(\cH)\cup Q)$.

By our assumption, there is an index $i(\star)\in [t]$ such that $\cH[V']$ 
is $\cG_{i(\star)}$-colorable. Choose a set 
$S\subseteq Q'$ of maximum size such that $\cH[V'\cup S]$ is still $\cG_{i(\star)}$-colorable.
If $S=Q'$ we are done, so suppose for the sake of contradiction that there exists a vertex 
$v\in Q'\setminus S$.

Due to the maximality of $S$ the triple system $\cH'=\cH[V'\cup S \cup\{v\}]$ 
is not $\cG_{i(\star)}$-colorable. On the other hand, Lemma~\ref{LEMMA:size-Z-H}~\ref{it:42a} 
and $|Q| \le 2\epsilon^{1/2}n$ entail
\begin{align*}
\delta(\cH')
& > (3\lambda_t-2\epsilon^{1/2})n^2 - |Z(\cH) \cup Q| n
  > (3\lambda_t-5\epsilon^{1/2})n^2 \\
  \text{and} \qquad 
|\cH'|
& > (\lambda_t-\epsilon)n^3 - |Z(\cH) \cup Q| n^2
  > (\lambda_t-4\epsilon^{1/2})n^3. 
\end{align*}

So by Lemma~\ref{LEMMA:extend-coloring} and $\zeta = 5\epsilon^{1/2}$ 
the $\cG_{i(\star)}$-colorability of $\cH'-v=\cH[V'\cup S]$ implies 
that~$\cH'$ itself is $\cG_{i(\star)}$-colorable as well. This contradiction 
completes the proof of Lemma~\ref{LEMMA:extend-coloring-equivalent-class}
\end{proof}

It remains to deduce Theorem~\ref{THM:H-Z-is-Gi-colorable}. The argument
involves the following invariant of $3$-graphs: Given a $3$-graph $\cH$ with 
equivalence classes $C_1, \ldots, C_m$ we set $\Psi(\cH)=\sum_{i=1}^m|C_i|^2$.

\begin{proof}[Proof of Theorem~\ref{THM:H-Z-is-Gi-colorable} using 
Lemma~\ref{LEMMA:extend-coloring-equivalent-class}]
Let $\eps$ be the constant delivered by Lemma~\ref{LEMMA:extend-coloring-equivalent-class}
and fix a sufficiently large natural number $n$. 
Assuming that the conclusion of Theorem~\ref{THM:H-Z-is-Gi-colorable} fails for our values 
of $\eps$ and $n$ we pick a counterexample $\cH$ such that the pair $(|\cH|, \Psi(\cH))$ 
is lexicographically maximal. Let $C_1, \ldots, C_m$ be the equivalence classes of $\cH$.

Recall that Lemma~\ref{LEMMA:size-Z-H}~\ref{it:42a} tells us $|Z_\eps(\cH)|\le \eps^{1/2}n$.
Since $\cH$ is a counterexample, it cannot be $\cG_i$-colorable for any $i\in[t]$.
Moreover,~\eqref{eq:lambdat} yields 
\[
|\cH|>(\lambda_t-\eps)n^3\ge(5/32-1/16)n^3=3n^3/32>2n^3/27
\]
and thus $\cH$ cannot be semibipartite. So by Lemma~\ref{LEMMA:complete-shadow-imply-Gi-colorable}
there exist two equivalence classes, say $C_1$ and $C_2$, such that
$\partial \cH$ possesses no edges from $C_1$ to $C_2$.
We may assume that $(d_{\cH}(C_1), |C_1|)\le_{\mathrm{lex}} (d_{\cH}(C_2), |C_2|)$,
where $\le_{{\rm lex}}$ indicates the lexicographic ordering on $\NN^2$.

Pick arbitrary vertices $v_1\in C_1$ and $v_2\in C_2$ and symmetrize only them.
That is, we let~$\cH'$ be the $3$-graph with $V(\cH')=V(\cH)$,
$\cH'-v_1 = \cH-v_1$ and $L_{\cH'}(v_1)=L_{\cH}(v_2)$.
Clearly, if $d_{\cH}(v_1)<d_{\cH}(v_2)$, then $|\cH'|>|\cH|$.
Moreover, if $d_{\cH}(v_1)=d_{\cH}(v_2)$, then $|\cH'|=|\cH|$, $|C_1|\le |C_2|$, and
\begin{align}
\Psi(\cH')-\Psi(\cH) \ge (|C_1|-1)^2+(|C_2|+1)^2-|C_1|^2-|C_2|^2=2(|C_2|-|C_1|+1)\ge 2. \notag
\end{align}
In both cases $(|\cH'|, \Psi(\cH'))$ is lexicographically larger than $(|\cH|, \Psi(\cH))$
and our choice of~$\cH$ implies that $\cH'-Z_\eps(\cH')$ is $\cG_i$-colourable for some $i\in [t]$.
By Lemma~\ref{LEMMA:size-Z-H}~\ref{it:42a} the set $Q=Z_\eps(\cH')\cup\{v_1\}$ has size $|Q|\le \epsilon^{1/2} n+1 < 2\epsilon^{1/2} n$.
Since the hypergraph $\cH-Q=\cH'-Q$ is $\cG_i$-colourable, Lemma~\ref{LEMMA:extend-coloring-equivalent-class}
implies that $\cH-Z(\cH)$ is $\cG_i$-colourable too. This contradiction to the choice of $\cH$ 
establishes Theorem~\ref{THM:H-Z-is-Gi-colorable}.
\end{proof}

\subsection{Transversals} \label{subsec:42}
Roughly speaking, the hypergraph $\cH-v$ appearing in Lemma~\ref{LEMMA:extend-coloring} arises 
from an almost balanced blow-up of $\cG_i$ by deleting a small number of edges. When we randomly 
select one vertex from each partition class of $\cH-v$ it is thus very likely that the resulting 
transversal induces a copy of $\cG_i$. In the proof of Lemma~\ref{LEMMA:extend-coloring} 
there are several places where we argue similarly in situations where some vertices from 
the transversals have been selected in advance. The precise statement we shall use in these cases
is Lemma~\ref{LEMMA:greedily-embedding-Gi} below. 

Consider a 3-graph with $V(\cG)=[m]$ and pairwise disjoint sets $V_1,\ldots,V_{m}$.
The blow-up $\cG[V_1,\ldots,V_{m}]$ of $\cG$ 
is obtained from $\cG$ by replacing each vertex $j\in[m]$ with the set $V_j$ and
each edge $\{j_1,j_2,j_3\}\in \cG$ with the complete $3$-partite $3$-graph 
with vertex classes $V_{j_1}$, $V_{j_2}$, and $V_{j_3}$.
For a $3$-graph $\cH$ we say that
a partition $V(H) = \bigdcup_{j\in[m]}V_j$ is a \emph{$\cG$-coloring} of $\cH$
if $\cH \subseteq \cG[V_1,\ldots,V_{m}]$.

\begin{lemma}\label{LEMMA:greedily-embedding-Gi}
Fix a real $\eta \in (0, 1)$ and integers $m, n\ge 1$.
Let $\cG$ be a $3$-graph with vertex set~$[m]$ and let $\cH$ be a further $3$-graph
with $v(\cH)=n$. 
Consider a vertex partition $V(\cH) = \bigdcup_{i\in[m]}V_i$ and the associated 
blow-up $\widehat{\cG} = \cG[V_1,\ldots,V_{m}]$ of $\cG$.
If two sets $T \subseteq [m]$ and $S\subseteq \bigcup_{j\not\in T}V_j$
have the properties
\begin{enumerate}[label=\alabel]
\item\label{it:47a} $|V_{j}| \ge (|S|+1)|T|\eta^{1/3} n$  for all $j \in T$, 
\item\label{it:47b} $|\cH[V_{j_1},V_{j_2},V_{j_3}]| \ge |\widehat{\cG}[V_{j_1},V_{j_2},V_{j_3}]| 
		- \eta n^3$ for all $\{j_1,j_2,j_3\} \in \binom{T}{3}$, 
\item\label{it:47c} and $|L_{\cH}(v)[V_{j_1},V_{j_2}]| \ge |L_{\widehat{\cG}}(v)[V_{j_1},V_{j_2}]| 
		- \eta n^2$ for all $v\in S$ and $\{j_1,j_2\} \in \binom{T}{2}$,
\end{enumerate}
then there exists a selection of vertices $u_j\in V_j$ for all $j\in [T]$
such that $U = \{u_j\colon j\in T\}$ satisfies
$\widehat{\cG}[U] \subseteq \cH[U]$ and
$L_{\widehat{\cG}}(v)[U] \subseteq L_{\cH}(v)[U]$ for all $v\in S$.
In particular, if $\cH \subseteq \widehat{\cG}$,
then $\widehat{\cG}[U] = \cH[U]$ and
$L_{\widehat{\cG}}(v)[U] = L_{\cH}(v)[U]$ for all $v\in S$.
\end{lemma}
\begin{proof}[Proof of Lemma \ref{LEMMA:greedily-embedding-Gi}]

Choose for $j\in T$ the vertices $u_j\in V_j$ independently and uniformly at random
and let $U = \{u_j\colon j\in T\}$ be the random transversal consisting of these vertices.
By~\ref{it:47a} and~\ref{it:47b} we have
\begin{align}
\PP\left(\{u_{j_1}, u_{j_2}, u_{j_3}\} \not\in \cH\right)
= 1 - \frac{|\cH[V_{j_1},V_{j_2},V_{j_3}]|}{|V_{j_1}||V_{j_2}||V_{j_3}|}
 \le \frac{\eta n^3}{|V_{j_1}||V_{j_2}||V_{j_3}|}
 \le \frac{1}{(|S|+1)^3|T|^3} \notag
\end{align}
for all edges $\{j_1,j_2,j_3\}\in \cG$. Similarly~\ref{it:47a} and~\ref{it:47c} lead to   
\begin{align*}
\PP\left(\{u_{j_1}, u_{j_2}\} \not\in L_{\cH}(v)\mid \{u_{j_1}, u_{j_2}\} 
	\in L_{\widehat{\cG}}(v)\right)
&= 1- \frac{|L_{\cH}(v)[V_{j_1},V_{j_2}]|}{|V_{j_1}||V_{j_2}|}
\le \frac{\eta n^2}{|V_{j_1}||V_{j_2}|} \\
&\le \frac{\eta^{1/3}}{(|S|+1)^2|T|^2} 
\end{align*}
for all $v\in S$ and all distinct $j_1,j_2\in [m]$.
Therefore, the union bound reveals
\begin{align}
\PP\left(\widehat{\cG}[U] \not\subseteq \cH[U]\right)
& \le \binom{|T|}{3}\frac{1}{(|S|+1)^3|T|^3}
< \frac{1}{6} \notag\\
{\rm and} \qquad \PP\left(L_{\widehat{\cG}}(v) \not\subseteq L_{\cH}(v)\right)
& \le  \binom{|T|}{2}\frac{\eta^{1/3}}{(|S|+1)^2|T|^2}
 <\frac{1}{2(|S|+1)}
\quad \text{ for every $v\in S$}. \notag
\end{align}
Altogether, the probability that $U$ fails to have the desired properties 
is at most 
\begin{align*}
\frac{1}{6} + \frac{|S|}{2(|S|+1)}<\frac 23.
\end{align*}
So the probability that $U$ has these properties is positive. 
\end{proof}

In practice the sets $U$ obtained by means of Lemma~\ref{LEMMA:greedily-embedding-Gi}
will be $2$-covered and thus they will be cores of some 
subgraphs $F\in \widehat{\cK}^3_{|U|}$ of $\cH$. In such situations $F$ will be $\cM_t$-free
and in order to exploit this fact we need to know that for $i \ne j$ the triple system
$\cG_i$ is in some sense far from being $\cG_j$-colorable
(see Lemma~\ref{LEMMA:cannot-embed-large-subgraph} below). The verification of this statement 
requires that we take a closer look into Construction~\ref{con:Gi} and the observation that follows 
summarizes everything we need in the sequel. 

\begin{observation}\label{OBS:link-Gi}
The triple systems $\cG_1,\ldots,\cG_{t}$ have the following properties.
\begin{enumerate}[label=\alabel]
\item\label{it:45a} For $i\in[t]$ and $v \in \cG_{i}$ the clique 
	number $\omega\left(L_{\cG_i}(v)\right)$
   of the link graph $L_{\cG_i}(v)$ satisfies
\begin{align}
\frac{n_i-1}{k_i-1}- \frac{k_i}{2} \le \omega\left(L_{\cG_i}(v)\right) \le \frac{n_i-1}{k_i-1}. \notag
\end{align}
\item\label{it:45b} 
We have 
\[
\frac{n_i-1}{k_i-1} - \frac{n_{i+1}-1}{k_{i+1}-1}>\frac{Q}{k_{i}^2}
\]
for every $i\in [t-1]$, where
\[
Q = \frac{n_1}{k_1+1} = \cdots = \frac{n_{t}}{k_t+1}\ge 2k_t^3\ge 16.
\]
\item\label{it:45c} For $i\in [t]$ the $3$-graph $\cG_i$ is regular with degree
            $3\lambda_t n_i^2$
            and 
            \[
            7n_i/8\le n_i-3k_i/2 \le \delta_2(\cG_i) \le \Delta_2(\cG_i) \le n_i-k_i.
            \]
\end{enumerate}
\end{observation}

\begin{proof}
	Part~\ref{it:45a} follows from the fact that due 
	to $\cG_i=(K^3_{n_i}\setminus H(\cD_i))\setminus \cS_i$ the link $L_{\cG_i}(v)$
	arises from an $((n_i-1)/(k_i-1))$-partite Tur\'an graph by the deletion of $k_i/2$
	edges. The proof of part~\ref{it:45c} is similar. For part~\ref{it:45b} it suffices to calculate 
	\begin{align*}
		\frac{n_i-1}{k_i-1} - \frac{n_{i+1}-1}{k_{i+1}-1}
		&= 
		\frac{Q(k_i+1)-1}{k_i-1}-\frac{Q(k_{i+1}+1)-1}{k_{i+1}-1} \\
		&=
		(2Q-1)\left(\frac{1}{k_i-1}-\frac{1}{k_{i+1}-1}\right) \\
		&\ge  Q\left(\frac{1}{k_i-1}-\frac{1}{k_i}\right)
		> \frac{Q}{k_{i}^2}. \qedhere
	\end{align*}
\end{proof}

As indicated earlier, this has the following consequence. 

\begin{lemma}\label{LEMMA:cannot-embed-large-subgraph}
If $i\in [t]$ and the triple system $\cG_i'$ arises from $\cG_i$ by the deletion of at most 
$Q/(2k_i^2)$ vertices, then $\cG'_i$ fails to be $\cG_j$-colorable for every $j\in [t]\setminus\{i\}$. 
\end{lemma}

\begin{proof}[Proof of Lemma \ref{LEMMA:cannot-embed-large-subgraph}]
Suppose first that $j\in [i-1]$. Due to 
\begin{align}
\delta_2(\cG'_i) 
\ge 
\delta_2(\cG_i) - \frac{Q}{2k_i^2}
\ge 1 \notag
\end{align}
we know that $\cG'_i$ is $2$-covered. Together with 
\[
v(\cG'_i)
=
n_i-Q/(2k_i^2)
>Q(k_i+1)-Q
\ge 
Q(k_j+1)
=n_j
\]
it follows that $\cG'_i$ is indeed not $\cG_j$-colorable.

If $j\in (i, t]$ we take an arbitrary vertex $v\in V(\cG'_i)$.  
The parts~\ref{it:45a} and~\ref{it:45b} of Observation~\ref{OBS:link-Gi} yield
\begin{align}
\omega\left(L_{\cG'_i}(v)\right)
 \ge \omega\left(L_{\cG_i}(v)\right) - \frac{Q}{2k_i^2}
 \ge \frac{n_i-1}{k_i-1} - \frac{k_i}{2} - \frac{Q}{2k_i^2}
> \frac{n_{j}-1}{k_j-1}. \notag
\end{align}
On the other hand, by Observation~\ref{OBS:link-Gi}~\ref{it:45a} again, any $\cG_j$-coloring 
of $\cG'_i$ would show that 
\[
\omega\left(L_{\cG'_i}(v)\right)
\le 
\omega\left(L_{\cG_j}(v)\right)
\le
\frac{n_{j}-1}{k_j-1}.\qedhere
\] 
\end{proof}

On most occasions the following corollary of Lemma~\ref{LEMMA:cannot-embed-large-subgraph}
will suffice. 

\begin{corollary}\label{c:1344}
	If $i\in [t]$, the $3$-graph $\cH$ is $\cM_t$-free and $U\subseteq V(\cH)$
	denotes a $2$-covered set of size $n_i+1$, then $\cH[U]$ is $\cG_i$-free.
\end{corollary}

\begin{proof}
	Assume for the sake of contradiction that $\cH[U]$ has a subgraph isomorphic to $\cG_i$.
	If~$i<t$ we can take a subgraph $F\in \widehat{\cK}^3_{n_i+1}$ of $\cH$ with $F[U]=\cH[U]$
	having $U$ as a core. As~$\cH[U]$ contains a copy of $\cG_i$, we have $\tau(F[U])\ge 2$.
	Now $F\not\in\cM_{t}$ implies that $F$ is $\cG_j$-colorable for some $j\in [t]$. 
	In particular, $\cG_i$ is $\cG_j$-colorable and by Lemma~\ref{LEMMA:cannot-embed-large-subgraph} 	this leads to $i=j$. In other words, $F$ is $\cG_i$-colorable, contrary to the fact 
	that $\partial F$ contains a copy of~$K_{n_i+1}$.
	
	It remains to discuss the case $i=t$. Now Lemma~\ref{LEMMA:subgraph-is-the-expansion-of-clique}
	yields a subgraph $F'$ of $F$ which belongs to $\cK^3_{n_t+1}$, and whose induced 
	subgraph on its core has covering number at least $2$. By Lemma~\ref{l:6521} this 
	contradicts~$\cH$ being $\cM_t$-free. 
\end{proof}

\subsection{Proof of the main lemma}\label{subsec:43}
This entire subsection is devoted to the proof of Lemma~\ref{LEMMA:extend-coloring}.
Select constants $\zeta$ and $N_0$ fitting into the hierarchy 
\[
	N_0^{-1}\ll\zeta\ll n_t^{-1}.
\]
Consider an $\cM_t$-free $3$-graph $\cH$ on $n\ge N_0$ vertices satisfying 
$|\cH|\ge (\lambda_t-\zeta)n^3$ and $\delta(\cH)\ge (3\lambda-\zeta)n^2$
such that for some $v\in V(\cH)$ and $i\in [t]$ the $3$-graph $\cH_{v} = \cH\setminus\{v\}$
is $\cG_i$-colorable. Set $V=V(\cH)$ and fix a partition $\bigdcup_{i\in[n_i]}V_{i} = V \setminus\{v\}$
exemplifying the $\cG_i$-colorability of $\cH_{v}$.
We divide the argument that follows into three main parts each of which consists of 
several claims.  

\smallskip

\noindent {\bf Part I. Analysis of $\cH_v$.} The three claims that follow only deal with $\cH_v$ 
but say nothing about~$v$ and its link.  
\begin{claim}\label{CLAIM:size-Vi}
We have $|V_j| = n/n_i \pm 5\zeta^{1/2} n$ for every $j \in [n_i]$.
\end{claim}

\begin{proof}[Proof of Claim \ref{CLAIM:size-Vi}]
Set $x_j = |V_j|/(n-1)$ for every $j \in [n_i]$.
By Proposition~\ref{prop:lagrange} (and the proof of Lemma~\ref{LEMMA:blowup-lagrangian})
we obtain
\begin{align}
|\cH_{v}|
= L_{\cG_i}(x_1,\ldots,x_{n_i})(n-1)^3
\le \left( \lambda_t - \frac{1}{9}\sum_{j\in[n_i]}\left(x_j-\frac1{n_i}\right)^2\right) n^3.\notag
\end{align}
Combined with 
\[
|\cH_v| \ge (\lambda_t-\zeta)n^3 - d_{\cH}(v) > (\lambda_t-2\zeta)n^3
\]
this leads to $\frac{1}{9}\sum_{j\in[n_i]}\left(x_j-1/n_i\right)^2 \le 2\zeta$,
whence $x_{j} = 1/n_i \pm (18\zeta)^{1/2}$ and 
\[
\big||V_j|-n/n_i\big|
\le (n-1)\big|x_j-1/n_i\big|+1/n_i\le (18\zeta)^{1/2}n+1/n_i\le 5\zeta^{1/2}n. \qedhere
\]
\end{proof}
Recall that the sets $V_1, \dots, V_{n_i}$ have been chosen in such a way that 
$\cH_v$ is a subgraph of the blow-up $\widehat{\cG}_{i} = \cG_i[V_1,\ldots,V_{n_i}]$
of~$\cG_i$. Our next objective is to compare the links of an arbitrary vertex $u \in V\setminus\{v\}$
in~$\cH_v$ and in~$\widehat{\cG}$. As a consequence of $\cH_{v} \subseteq \widehat{\cG}_{i}$
we know $L_{\cH_{v}}(u) \subseteq L_{\widehat{\cG}_{i}}(u)$
and $|L_{\cH_{v}}(u)| \le  |L_{\widehat{\cG}_{i}}(u)|$.
Members of $L_{\widehat{\cG}_{i}}(u) \setminus L_{\cH_{v}}(u)$ are referred to as the
{\it missing pairs of $u$}. By Lemma~\ref{LEMMA:blowup-lagrangian} the global number of 
missing edges can be bounded from above by 
\begin{equation}\label{eq:4241}
	\big|\widehat{\cG}_i\setminus\cH_v\big|
	\le
	\lambda_t(n-1)^3-(\lambda_t-\zeta)n^3+d_\cH(v)
	\le
	2\zeta n^3.
\end{equation}
Locally we obtain the following. 
\begin{claim}\label{CLAIM:capacity-of-u}
Every $u \in V\setminus\{v\}$ satisfies
$|L_{\widehat{\cG}_i}(u)| < (3\lambda_t + 6n_i\zeta^{1/2})n^2$. 
Moreover the number of missing pairs of $u$ is bounded by 
$|L_{\widehat{\cG}_i}(u) \setminus L_{\cH_v}(u)|<7\zeta^{1/2}n_in^2$.
\end{claim}
\begin{proof}[Proof of Claim \ref{CLAIM:capacity-of-u}]
Since $\cG_i$ is $(3\lambda_tn_i^2)$-regular, Claim \ref{CLAIM:size-Vi} yields
\begin{align}
\big|L_{\widehat{\cG}_i}(u)\big|
  \le 3\lambda_t n_i^2\left(\frac{n}{n_i}+5\zeta^{1/2}n\right)^2
 =   3\lambda_t n^2 \left(1+5\zeta^{1/2}n_i\right)^2 
 <   (3\lambda_t + 6\zeta^{1/2}n_i)n^2, \notag
\end{align}
where we used $\lambda_t<1/6$ and our hierarchy $\zeta\ll n^{-1}_i$. Owing to the minimum 
degree condition $\delta(\cH)\ge(3\lambda_t-\zeta)n^2$ this entails the upper bound 
\begin{align}
\big|L_{\widehat{\cG}_i}(u) \setminus L_{\cH_v}(u)\big|
\le \left(3\lambda_tn^2 + 6\zeta^{1/2}n_in^2\right) - \left(3\lambda_t n^2 - \zeta n^2 - n\right)
< 7\zeta^{1/2}n_in^2 \notag
\end{align}
on the number of missing pairs of $u$. 
\end{proof}

It can now be shown that in $\cH_v$ all neighborhoods have roughly the expected 
size $\frac{n_i-1}{n_i}n$, but for our concerns it suffices to establish a
lower bound.
\begin{claim}\label{CLAIM:size-Nu-Hv}
We have $|N_{\cH_v}(u)| \ge \frac{n_i-1}{n_i}n - 17 \zeta^{1/2} n_i n$ for 
every $u \in V\setminus\{v\}$.
\end{claim}

\begin{proof}[Proof of Claim \ref{CLAIM:size-Nu-Hv}]
Let $j\in [n_i]$ be the index satisfying $u\in V_j$. Since every vertex in 
$V\setminus (V_j\cup N_{\cH_v}(u)\cup\{v\})$
belongs to at least $\delta_2(\cG)\cdot \min\{|V_\ell|\colon \ell\in [n_i]\}$ missing pairs of $u$,
and every missing pair is counted at most twice in this manner, Claim~\ref{CLAIM:capacity-of-u} yields
\[
	\big|V\setminus (V_j\cup N_{\cH_v}(u)\cup\{v\})\big|
	\cdot\delta_2(\cG)
	\cdot\min\bigl\{|V_\ell|\colon \ell\in [n_i]\bigr\}
	<
	14\zeta^{1/2}n_in^2.
\]
So by Observation~\ref{OBS:link-Gi}~\ref{it:45c} and Claim~\ref{CLAIM:size-Vi} 
the assumption $|N_{\cH_v}(u)| < \frac{n_i-1}{n_i}n - 17 \zeta^{1/2}n_i n$ would yield 
the contradiction
\begin{align}
\left(17 \zeta^{1/2}n_i n - 5\zeta^{1/2}n-1\right)
\cdot\frac{7n_i}8\cdot\left(\frac{n}{n_i}-5\zeta^{1/2}n\right)
< 14\zeta^{1/2}n_in^2.\notag
\end{align}
Thereby Claim~\ref{CLAIM:size-Nu-Hv} is proved.
\end{proof}



\smallskip
\noindent{\bf Part II. Choice of a vertex class for $v$.} Our strategy for showing that $\cH$
is $\cG_i$-colorable is to adjoin $v$ to one the partition classes $V_1, \dots, V_{n_t}$.
In fact, there is only one of these classes $v$ fits into. Before finding this class we show 
a statement that has to hold if our plan is sound.  

\begin{claim}\label{CLAIM:Vi-indep-in-Lv}
We have $L_{\cH}(v) \cap \binom{V_j}{2} = \emptyset$ for every $j \in [n_i]$.
\end{claim}
\begin{proof}[Proof of Claim \ref{CLAIM:Vi-indep-in-Lv}]
Without loss of generality we may assume that $j = 1$.
Let $u_0, u_1\in V_1$ be two distinct vertices. By Lemma \ref{LEMMA:greedily-embedding-Gi}
applied to $S=\{u_0, u_1\}$ and $T=[2, n_i]$ there exist vertices 
$u_j \in V_j$ for $j\in [2,  n_i]$ such that the subgraphs of $\cH$ 
induced by $\{u_0,u_2,\ldots,u_{n_i}\}$
and $\{u_1,u_2,\ldots,u_{n_i}\}$ are isomorphic to $\cG_i$.
Now Corollary~\ref{c:1344} informs us that the set $U=\{u_0, u_1, \ldots, u_{n_t}\}$
cannot be $2$-covered, for which reason $u_0u_1\not\in\partial \cH$.
So, in particular, we have $u_0u_1\not\in L_{\cH}(v)$. 
\end{proof}
\begin{claim}\label{CLAIM:v-has-few-neighbors-in-some-Vj}
There exists $j \in [n_i]$ such that $|N_{\cH}(v) \cap V_j| < \zeta^{1/7} n$.
\end{claim}
\begin{proof}[Proof of Claim \ref{CLAIM:v-has-few-neighbors-in-some-Vj}]
Suppose for the sake of contradiction that the sets $W_j=N_{\cH}(v) \cap V_j$ satisfy 
$|W_j| \ge \zeta^{1/7} n$ for every $j \in [n_i]$.
Applying Lemma \ref{LEMMA:greedily-embedding-Gi} to $W_j$ here in place of $V_j$ there 
and to $S = \emptyset$, $T = [n_i]$ we obtain vertices 
$u_j \in V_j$ for all $j\in[n_i]$
such that the set $U=\{u_1,\ldots,u_{n_i}\}$
induces a copy of $\cG_i$ in $\cH$.  But now the $2$-covered set $U\cup\{v\}$ contradicts 
Corollary~\ref{c:1344}.
\end{proof}

It will turn out later that the index $j$ delivered by 
Claim~\ref{CLAIM:v-has-few-neighbors-in-some-Vj} is unique. 
Without loss of generality we may assume that 
\begin{equation}\label{eq:2354}
|N_{\cH}(v) \cap V_1| < \zeta^{1/7} n.
\end{equation}

\smallskip
\noindent{\bf Part III. The link of $v$.} 
It remains to show that $L_{\cH}(v) \subseteq L_{\widehat{G}_{i}}(V_1)$. To this end we define
\begin{align}
N_{v}(u) = \left\{j\in [n_i]\colon |N_\cH(u, v)\cap V_j| \ge \zeta^{1/7} n\right\} \notag
\end{align}
for every $u \in N_{\cH}(v)$ .
The upper bound on $\Delta_2(\cG_i)$ in Observation~\ref{OBS:link-Gi}~\ref{it:45c} transfers 
to these sets as follows.

\begin{claim}\label{CLAIM:max-deg-at-most-n-k-in-Lv}
We have $|N_{v}(u)| \le n_i-k_i$ for every $u \in N_{\cH}(v)$.
\end{claim}
\begin{proof}[Proof of Claim \ref{CLAIM:max-deg-at-most-n-k-in-Lv}]
Assume for the sake of contradiction that there is a set $N_\star\subseteq N_v(u)$
such that $|N_\star|=n_i-k_i+1<n_i-2$. 
As in the proof of Claim~\ref{CLAIM:v-has-few-neighbors-in-some-Vj} there exist 
ver\-tices $u_j\in N_\cH(u, v)\cap V_j$ for $j\in N_\star$ such that 
$\cG_i[N_\star]$ is isomorphic to $\cH[U]$, where $U=\{u_j\colon j\in N_\star\}$.

Now we consider the $3$-graph $F=\cH[U\cup \{u, v\}]$. Clearly $U\cup\{u, v\}$
is $2$-covered in $F$ and $\tau(F)\ge \tau(\cG_i[N_\star])\ge 2$. So $F\not\in \cM_t$
tells us that $F$ is $\cG_s$-colorable for some $s\in [t]$. 

On the other hand by Lemma \ref{LEMMA:cannot-embed-large-subgraph} and 
$|U| \ge n_i-k_i+2 > n_i-Q/(2k_i^2)$ the subgraph $F[U]$ of $F$ cannot be 
$\cG_s$-colorable for any $s\in [t]\setminus \{i\}$.  

Summarizing this discussion, $F$ is $\cG_i$-colorable. As $F$ is also $2$-covered,
$F$ is actually isomorphic to a subgraph of $\cG_i$ and, consequently, 
$n_i-k_i<|N_\star|=d_F(u, v)\le \Delta_2(\cG_i)$, contrary to 
Observation~\ref{OBS:link-Gi}~\ref{it:45c}.
\end{proof}
\begin{claim}\label{CLAIM:size-Sg-is-ni-1}
We have $|N_{\cH}(v) \cap V_j| \ge \zeta^{1/7} n$ for every $j\in [2, n_i]$.
\end{claim}

\begin{proof}[Proof of Claim \ref{CLAIM:size-Sg-is-ni-1}]
The minimum degree condition imposed on $\cH$ and $6\lambda_t=1-\frac{k_i+1}{n_i}$
yield
\begin{align*}
	\left(1-\frac{k_i+1}{n_i}-2\zeta\right)n^2
	=
	2(3\lambda_t-\zeta)n^2
	\le 
	2d_\cH(v)
	\le
	\Delta\bigl(L_\cH(v)\bigr)|N_\cH(v)|.
\end{align*}   
Claim \ref{CLAIM:max-deg-at-most-n-k-in-Lv} allows us to bound the first factor on the right side
from above by  
\begin{align}
\Delta\bigl(L_\cH(v)\bigr)
\le 
(n_i-k_i)\left(\frac{n}{n_i}+5\zeta^{1/2}n\right) + k_i \zeta^{1/7} n
< \frac{n_i-k_i}{n_i}n + 2k_i \zeta^{1/7} n. \notag
\end{align}
Altogether we obtain 
\[
	\frac{n_i-(k_i+1)-2n_i\zeta}{(n_i-k_i)+2k_in_i \zeta^{1/7}}\le \frac{|N_\cH(v)|}n,
\]
which due to 
\[
	\frac{n_i-(k_i+1)}{n_i-k_i}=1-\frac 1{n_i-k_i}>1-\frac{5/4}{n_i}
\]
and $\zeta\ll n_i^{-1}$ implies
\[
	\left(1-\frac{3/2}{n_i}\right)n \le |N_\cH(v)|.
\]
On the other hand, setting
$I = \left\{j\in[2, n_i]\colon|N_{\cH}(v) \cap V_j| \ge \zeta^{1/7} n \right\}$
Claim~\ref{CLAIM:size-Vi} and~\eqref{eq:2354} lead to
\[	
	|N_\cH(v)|
	\le 
	|I|\left(\frac 1{n_i}+5\zeta^{1/2}\right)n+\zeta^{1/7}n_in.
\]
Combining both estimates we arrive at $|I|>n_i-7/4$, whence $I = [2, n_i]$.
\end{proof}

\begin{claim}\label{CLAIM:v-has-bo-neighbor-in-V1}
We have $N_{\cH}(v) \cap V_1 = \emptyset$.
\end{claim}
\begin{proof}[Proof of Claim \ref{CLAIM:v-has-bo-neighbor-in-V1}]
Suppose that there exists $u_1 \in N_{\cH}(v) \cap V_1$.
Owing Claim~\ref{CLAIM:size-Sg-is-ni-1} we can apply Lemma \ref{LEMMA:greedily-embedding-Gi} 
with $S = \{u_1\}$ and  $T = [2, n_i]$ in order to obtain vertices 
$u_j \in N_{\cH}(v) \cap V_j$ for $j\in [2, n_i]$ such that $\cH$ induces a copy 
of $\cG_i$ on $U=\{u_1,\ldots,u_{n_i}\}$. Since $U\cup\{v\}$ is $2$-covered, this 
contradicts Corollary~\ref{c:1344}. 
\end{proof}
Let us recall that $L_{\widehat{\cG}_i}(V_1)$ denotes the common $\widehat{\cG}_i$-link of 
all vertices in $V_1$.

\begin{claim}\label{CLAIM:link-v-is-correct}
We have $L_{\cH}(v) \subseteq L_{\widehat{\cG}_i}(V_1)$.
\end{claim}
\begin{proof}[Proof of Claim \ref{CLAIM:link-v-is-correct}]
Due to the Claims~\ref{CLAIM:Vi-indep-in-Lv} and~\ref{CLAIM:v-has-bo-neighbor-in-V1}
we know that $L_{\cH}(v)$ is an $(n_i-1)$-partite graph with vertex classes $V_2, \dots, V_{n_i}$.
So if Claim~\ref{CLAIM:link-v-is-correct} fails we may assume without loss of 
generality $123\not\in \cG_i$ and that there exists a pair $u_2u_3\in L_{\cH}(v)$ 
with $u_2\in V_2$, $u_3\in V_3$. 

Since $|V_1| > n/(2n_i)$ and $|N_{\cH}(v) \cap V_j| \ge \zeta^{1/7} n$ for $j\in [4, n_i]$,
Lemma \ref{LEMMA:greedily-embedding-Gi} applied to $S = \{u_2,u_3\}$ and $T = \{1,4,\ldots,n_i\}$ 
delivers vertices $u_1\in V_1$ and $u_j \in N_{\cH}(v) \cap V_j$
for $j\in [4, n_i]$ such that the set $U' = \{u_1,u_4,\ldots,u_{n_i}\}$
satisfies
\begin{equation}\label{eq:4003}
\cH[U'] = \widehat{\cG}_i[U']
\quad{\rm and}\quad
L_{\cH}(u_{\ell})[U'] = L_{\widehat{\cG}_{i}}(u_{\ell})[U'] \quad{\rm for}\quad\ell = 2,3. 
\end{equation}
Consider the set $U=\{u_1, \ldots, u_{n_i}\}$. Because of~\eqref{eq:4003} and $123\not\in \cG_i$
the map $i\longmapsto u_i$ is an embedding of $L_{\cG_i}(1)$ into $\cH$ and for this reason 
we have 
\begin{equation}\label{eq:4004}
	d_{\cH[U]}(u_1)\ge d_{\cG_i}(1).
\end{equation}

Next we choose for every $j\in [4, n_i]$ an edge $e_j\in \cH$ such that $u_j, v\in e_j$ 
and observe that~$U$ is $2$-covered in the $3$-graph
\begin{align}
F = \{vu_2u_3\} \cup \{e_j\colon 4\le j \le n_i\} \cup \cH[U]. \notag
\end{align}
Moreover, $|F| \le |\cG_i| + n_i-2 < \binom{n_i}{3}$ implies $F \in \widehat{\cK}_{n_i}^3$.
Since $F[U']=\cH[U']$ is isomorphic to $\cG_i-\{2, 3\}$, 
Lemma~\ref{LEMMA:cannot-embed-large-subgraph}
tells us that $F$ cannot be $\cG_j$-colorable for any $j\in [t]\setminus\{i\}$. But on the other
hand we have $\tau(F[U])\ge 2$ and $F\not\in\cM_{t}$, so altogether $F$ is $\cG_i$-colorable. 

Fix a homomorphism $\phi\colon V(F) \longrightarrow V(\cG_i)$ from $F$ to $\cG_i$. Since $U$ and 
$U_v=U\cup\{v\}\setminus \{u_1\}$ are $2$-covered subsets of $F$ whose size is $n_i=v(\cG_i)$,
the map $\phi$ has to be bijective on $U$ and $U_v$, which is only possible if $\phi(v)=\phi(u_1)$.
Now $\phi$ embeds the link $L_{F[U]}(u_1)$ into the link $L_{\cG_i}(\phi(u_1))$. Moreover, 
$vu_2u_3\in F$ implies that $\phi(u_2)\phi(u_3)$ belongs to the link $L_{\cG_i}(\phi(u_1))$
as well and by $123\not\in \cG_i$ this edge is not in the image $\phi(L_{F[U]}(u_1))$. 
Altogether this proves $d_{F[U]}(u_1)+1\le d_{\cG_i}(\phi(u_1))$, which in view of $F[U]=\cH[U]$
and~\eqref{eq:4004} contradicts the regularity of $\cG_i$.
\end{proof}

By Claim~\ref{CLAIM:link-v-is-correct}
the partition $\bigcup_{j\in[n_i]}\widehat{V}_j$, where 
\[
	\widehat{V}_j = 
	\begin{cases}
		V_1\cup\{v\} & \text{ if } j=1 \cr
		V_j & \text{ if } 2\le j \le n_i
	\end{cases}
\]
is a $\cG_i$-coloring of $\cH$.
This completes the proof of Lemma \ref{LEMMA:extend-coloring}.

\section{Feasible region of \texorpdfstring{$\cM_{t}$}{M} and \texorpdfstring{$\xi(\cM_t)$}{xi(M)}} 
\label{subsection:feasible}
 
We prove Theorem~\ref{THM:feasible-region-Mt} and that $\xi(\cM_t) = t$ in this section.
First, let us show a simple lemma.

\begin{lemma}\label{LEMMA:dense-Gi-colorable-3gp-shadow}
Suppose that $\cH$ is an $n$-vertex $\cG_i$-colorable $3$-graph for some $i\in [t]$. 
If $|\cH|\ge (\lambda_t-\epsilon)n^3$, 
then $|\partial\cH| \ge \bigl(\frac{n_i-1}{2n_i}-3\epsilon^{1/2}n_i\bigr)n^2$.
\end{lemma}
\begin{proof}[Proof of Lemma~\ref{LEMMA:dense-Gi-colorable-3gp-shadow}]
Let $V(\cH) = \bigcup_{j\in[n_i]}V_j$ be a $\cG_i$-coloring of $\cH$.
Now by Proposition~\ref{prop:lagrange}, $|V_j| = (1/n_i\pm 3\epsilon^{1/2})n$ for all $j\in[n_i]$.
Call a pair $\{u,v\}$ with $u\in V_{j}, v\in V_{k}$ and $j \neq k$ missing if $uv \not\in \partial\cH$,
and let $M$ denote the set of all missing pairs.
Since $\delta_2(\cG_i) \ge 7n_i/8$, we obtain
\begin{align}
|M| \cdot \frac{7n_i}8\cdot \left(\frac{1}{n_i}-3\epsilon^{1/2}\right)n
\le 3\epsilon n^3, \notag
\end{align}
which yields $|M| < 4\epsilon n^2$.
Therefore,
\[
|\partial\cH|
> \binom{n_i}{2} \times \left(\frac{1}{n_i}-3\epsilon^{1/2}\right)^2n^2 - |M|
> \frac{n_i-1}{2n_i} n^2 - 3\epsilon^{1/2}n_i n^2. \qedhere
\]
\end{proof}

We remark that the stronger conclusion 
$|\partial\cH| \ge \bigl(\frac{n_i-1}{2n_i}-5\epsilon n_i\bigr)n^2$
could be shown by arguing more carefully, but this is immaterial to what follows.  

\begin{proof}[Proof of Theorem~\ref{THM:feasible-region-Mt}]
Recall from Section~\ref{SEC:the-extremal-configurations} that semibipartite $3$-graphs 
are $\cM_t$-free. 
This yields ${\rm proj}\Omega(\cM_t) = [0,1]$, 
as for every $x\in [0, 1]$ there exists a good sequence of semibipartite $3$-graphs 
such that the edge densities of their shadows converges to $x$. 

Theorem~\ref{THM:t-stability}~\ref{it:11a} implies that $g(\cM_t,x) \le 6\lambda_t$ 
for all $x\in[0,1]$. Furthermore for every $i\in [t]$ the sequence of balanced blow-ups of $\cG_i$ 
shows the equality $g(\cM_t, 1-1/n_i) = 6\lambda_t$.  
So, in order to finish the proof
it suffices to show that if some $x\in [0, 1]$ satisfies
$g(\cM_t,x) = 6\lambda_t$, 
then there is an index $i\in [t]$ such that $x = 1-1/n_i$.

Fix such an $x \in [0,1]$ and let $\left(\cH_n\right)_{n=1}^{\infty}$ be a good sequence 
of $\cM_t$-free $3$-graphs realizing $(x,6\lambda_t)$.
Consider an arbitrary $\delta >0 $ and let $\epsilon > 0, N_0$ be the 
constants guaranteed 
by Theorem~\ref{THM:t-stability}~\ref{it:11b}. Without loss of generality we may assume 
$\eps\le \delta$. 
By our choice of~$\left(\cH_n\right)_{n=1}^{\infty}$ there exists $n_0\in\NN$ such that
\begin{align}
d(\cH_n) = 6\lambda_t \pm \epsilon
\quad{\rm and}\quad
d(\partial\cH_n) = x \pm \epsilon \notag
\end{align}
hold for all $n \ge n_0$.
By Theorem~\ref{THM:t-stability}~\ref{it:11b},
for every $n \ge \max\{n_0,N_0\}$ the $3$-graph $\cH_n$ is $\cG_i$-colorable for some $i=i(n)\in [t]$
after removing at most $\delta v(\cH_n)$ vertices.
Therefore,
\begin{align}
|\partial\cH_n| \le \left(\frac{n_i-1}{2n_i} + \delta \right)v(\cH_n)^2, \notag
\end{align}
and, on the other hand, by Lemma~\ref{LEMMA:dense-Gi-colorable-3gp-shadow},
\begin{align}
|\partial\cH_n|
> \left(\frac{n_i-1}{2n_i}-3\epsilon^{1/2}n_i\right)(1-\delta)^2v(\cH_n)^2 \notag
> \frac{n_i-1}{2n_i}v(\cH_n)^2 - \left(3\epsilon^{1/2}n_i+2\delta\right)v(\cH_n)^2. \notag
\end{align}
Summarizing and taking $\eps\le \delta$ into account we arrive at 
\begin{equation}\label{eq:5343}
\frac{n_i-1}{n_i}-\left(6\delta^{1/2}n_t+4\delta\right)
< d(\partial\cH_n)
\le \frac{n_i-1}{n_i}+ 2\delta,  
\end{equation}
where, let us recall, $i=i(n)$ might depend on $n$. So what~\eqref{eq:5343} means 
is that if we set 
\[
	I_i(\delta)=\left[\frac{n_i-1}{n_i}-6\delta^{1/2}n_i-4\delta, \frac{n_i-1}{n_i}+2\delta\right]
\] 
for every $i\in [t]$, then 
\[
	d(\partial\cH_n)\in I_1(\delta)\cup\dots\cup I_t(\delta)
\]
holds for every $n\ge n_0$. As the set on the right side is closed we obtain 
\[
	x\in  I_1(\delta)\cup\dots\cup I_t(\delta)
\]
in the limit $n\to\infty$. Since $\delta > 0$ was arbitrary, 
\[
	x\in \bigcap_{\delta>0} \bigl(I_1(\delta)\cup\dots\cup I_t(\delta)\bigr)
	=
	\bigl\{1-1/n_i\colon i\in [t]\bigr\}
\]
follows. 
\end{proof}

Recall that we already proved that $\cM_t$ is $t$-stable,
which, by definition, shows that $\xi(\cM_t) \le t$.
Therefore, in order to prove $\xi(\cM_t) = t$ it suffices to show that $\xi(\cM_t) \ge t$,
and this is an easy consequence of the following proposition and Theorem~\ref{THM:feasible-region-Mt}.

\begin{proposition}\label{PROP:number-of-maxima-and-stability-number}
Let $\cF$ be a family of $r$-graphs and let $M$ be the set of global maxima of~$g(\cF)$.
If $M$ is finite, then $|M| \le \xi(\cF)$.
\end{proposition}

The proof of this result involves the \emph{edit distance} of hypergraphs: Given two $r$-graphs~$H$
and $H'$ with the same number of vertices we set 
\[
	d_1(H, H')=\min\{|H\triangle H''|\colon V(H'')=V(H) \text{ and } H''\cong H'\}.
\]
It is well known and easy to confirm that this distance satisfies the triangle inequality.
  
\begin{proof}[Proof of Proposition~\ref{PROP:number-of-maxima-and-stability-number}]
If $\cF$ is degenerate, then $g(\cF)$ is the constant function whose value is always~$0$
and $M$ is infinite. So we may assume that the Tur\'an density $y=\pi(\cF)$ is positive. 
Let us write $M=\{(x_i, y)\colon i\in [m]\}$ such that $x_1<\dots<x_m$ and $m=|M|$. 
For every $i\in [m]$ we select a good sequence $\left(\cH_i(n)\right)_{n=1}^{\infty}$ 
of $\cF$-free $r$-graphs realizing $(x_i, y)$. Without loss of generality we have $v(\cH_i(n))=n$
for every positive integer $n$. 
Now suppose for the sake of contradiction that $t = \xi(\cF)$ is smaller than $m$.

\begin{claim}\label{c:5039}
	For every $\delta>0$ there are distinct $i, j\in [m]$ and $n>1/\delta$ such that  
	\[
	d_1(\cH_i(n), \cH_j(n))\le \delta n^r 
	\quad \text{ and } \quad 
	\min\{|\cH_i(n)|, |\cH_j(n)|\}\ge (y-\delta)\binom nr.
	\] 
\end{claim}

\begin{proof}[Proof of Claim~\ref{c:5039}]
	By the definition of $\xi(\cF)=t$ there are $n_0\in\NN$ and $\eps>0$ such that for 
	every $n\ge n_0$ there exists a family $\{\cG_1(n), \ldots, \cG_{t}(n)\}$ of $r$-graphs
	on $n$ vertices such that for every $\cF$-free $r$-graph $\cH$ with $v(\cH)=n$ 
	and $|\cH|\ge (y-\eps)\binom nr$ there is some $s\in [t]$ such 
	that $d_1(\cH, \cG_s(n))\le (\delta/2) n^r$ 
	As usual, we may suppose that $\eps\le \delta$.
	
	Now choose $n\ge n_0, \delta^{-1}$ such that for every $i\in [m]$ 
	we have $d(\cH_i(n))\ge y-\eps$. Stability allows us to select for every $i\in [m]$
	an index $s(i)\in [t]$ 
	such that $d_1(\cH_i(n), \cG_s(n))\le \delta n^r$. 
	By $t<m$ the map $i\longmapsto s(i)$ cannot be injective, i.e., there are distinct $i, j\in [m]$
	and $s\in [t]$ such that $s(i)=s(j)=s$. Now the triangle inequality yields 
	\[
		d_1(\cH_i(n), \cH_j(n))\le d_1(\cH_i(n), \cG_s(n))+d_1(\cG_s(n), \cH_j(n))\le \delta n^r,
	\]
	as desired.	 
\end{proof}

Notice that, as stated, Claim~\ref{c:5039} allows $i$ and $j$ to depend on $\delta$.
However, a quick thought reveals that there actually have to be two indices $i<j$ that work
for every $\delta>0$. Now we intend to contradict the finiteness of $M$ by 
proving $[x_i, x_j]\times\{y\}\subseteq M$. 

To this end, let $x\in [x_i, x_j]$ and a large integer $N$ be given. It suffices to construct 
an $\cF$-free $r$-graph $\cH$ satisfying $v(\cH)>N$, $d(\partial \cH)=x\pm 1/N$ and $d(\cH)=y\pm 1/N$. 
By Claim~\ref{c:5039} applied to $\delta\ll N^{-1}$ there is some $n>N$ such 
that $d_1(\cH_i(n), \cH_j(n))\le  \delta n^r$ and $\min\{|\cH_i(n)|, |\cH_j(n)|\}\ge (y-\delta)\binom nr$.
Assume without loss of generality that 
\[
|\cH_i(n)\triangle \cH_j(n)| \le \delta n^r.
\]
Now consider the following process transforming $\cH_i(n)$ into $\cH_i(n)$: Start with $\cH_i(n)$
and remove edges one by one until $\cH_i(n)\cap \cH_j(n)$ is reached. Then, keep adding edges 
one by one until you arrive at $\cH_j(n)$. Every $r$-graph occurring along the way is $\cF$-free. 
Moreover, since deleting or adding an edge can affect the size of the shadow by at most~$r$, 
in every step of the process the shadow density changer by at most $r/\binom n{r-1}$. 
Thus at some moment we 
pass an $r$-graph $\cH$ such that $|d(\partial \cH)-x|\le r/\binom n{r-1}\le \delta$. 
Finally, $d(\cH)\ge d(\cH_i(n)\cap \cH_j(n))\ge d(\cH_i(n))-|\cH_i(n)\triangle \cH_j(n)| /\binom nr
\ge y-O(\delta)$ completes the proof that $\cH$ has all desired properties.  
\end{proof}

\section{Concluding remarks}\label{SEC:remarks}
For every positive integer $t$ we constructed a family of
$3$-graphs $\{\cG_1,\ldots,\cG_t\}$ that have the same Lagrangian $\lambda_t$,
and we showed that there is a family $\cM_t$ of $3$-graphs whose extremal configurations 
are balanced blow-ups of
$\cG_1,\ldots,\cG_t$, and whose stability number is $\xi(\cM_t) = t$.
Notice that our choice of $\lambda_t$ is very close to $1/6$, which is the supremum of the 
Lagrangians of all $3$-graphs.
It would be interesting to find for every integer $t \ge 2$ the minimum value (if it exists) 
of $\lambda = \lambda(t)$ so that
there exists a $t$-stable family $\cF_t$ with $\pi(\cF_t) = 6\lambda$.
A result of Erd\H{o}s~\cite{E64} implies that there are no Tur\'an densities in the 
interval $(0, 2/9)$. This motivates the following question.  

\begin{problem}\label{PROB:2/9-implies-stable}
Does there exist a family $\cF$ of triple systems with $\pi(\cF) = 2/9$
but $\xi(\cF) \neq 1$?
\end{problem}

For a family $\cF$ of $r$-graphs let 
$M(\cF)=\{x\in {\rm proj}\Omega(\cF)\colon g(\cF)(x) = \pi(\cF))$ 
be the set of abscissae of the global maxima of its feasible 
region function. As we have shown here, $|M(\cF)|$ can be every finite cardinal except zero. 
In would be interesting to know whether $M(\cF)$ can be infinite and, in case the answer is 
affirmative, there immediately arise further questions. 

\begin{problem}\label{PROB:countably-many-maxima}
For $r\ge 3$ does there exist a non-degenerate family $\cF$ of $r$-graphs so that
$g(\cF)$ has infinitely many global maxima? If so, can the set $M(\cF)$ be uncountable?
Can it even contain a non-trivial interval? 
\end{problem}

Notice that if the last question on intervals has a negative answer, 
then in Proposition~\ref{PROP:number-of-maxima-and-stability-number} the assumption that $M$ 
should be finite can be omitted. In fact, it is somewhat bizarre that we do not know the following. 

\begin{problem}
	Let $\cF$ be a non-degenerate family of $r$-graphs such that $M(\cF)$ is infinite. 
	Can it nevertheless happen that $\cF$ has finite stability number? 
\end{problem}

In a forthcoming work~\cite{LMR3} we will show an extension of our results 
about triples systems to $r$-graphs for all $r\ge 4$
and exhibit a family $\cM_{t}^{r}$ 
that is $t$-stable such that the function~$g(\cM_{t}^{r})$
has exactly $t$-global maxima.


\end{document}